\documentclass[12pt]{article}
\usepackage{amsfonts,amssymb,amsmath,amsthm}
\usepackage{soul}
\usepackage{graphicx}
\usepackage{epstopdf}
\ifpdf
  \DeclareGraphicsExtensions{.eps,.pdf,.png,.jpg}
\else
  \DeclareGraphicsExtensions{.eps}
\fi
\usepackage{pbox}
\usepackage{color,url}
\usepackage{dsfont}
\usepackage{setspace}
\usepackage{graphicx}
\usepackage{epsfig}
\usepackage{float}
\usepackage{stmaryrd}
\usepackage{pifont}
\usepackage{comment}
\usepackage{enumerate}
\usepackage[section]{placeins}
\usepackage{xcolor}
\usepackage{mathtools, nccmath}

\newtheorem{thm}{\bf{Theorem}}[section]
\newtheorem{lem}[thm]{\bf{Lemma}}
\newtheorem{remark}[thm]{\bf{Remark}}

\newtheorem{df}[thm]{\bf{Definition}}

\newtheorem{ass}[thm]{\emph{Assumption}}

\newtheorem{prop}[thm]{\bf{Proposition}}

\newtheorem{ex}[thm]{\bf{Example}}

\newcommand{\dom}{\operatorname{dom}}

\newcommand{\1}{\operatorname{\mathds{1}}}

\newcommand{\Span}{\operatorname{span}}

\newcommand{\Id}{\operatorname{Id}}

\newcommand{\rank}{\operatorname{rank}}

\newcommand{\Diag}{\operatorname{Diag}}

\newcommand{\Proj}{\operatorname{Proj}}

\newcommand{\R}{\operatorname{\mathbb{R}}}
\newcommand{\N}{\operatorname{\mathbb{N}}}
\newcommand{\X}{\operatorname{\mathcal{X}}}

\newcommand{\y}{\operatorname{\mathcal{Y}}}
\newcommand{\m}{\operatorname{\mathcal{M}}}
\newcommand{\s}{\operatorname{\mathcal{S}}}

\newcommand{\bbm}{\begin{bmatrix}}
\newcommand{\ebm}{\end{bmatrix}}

\newcommand{\ds}{\delta_s}

\newcommand{\dsf}{\delta_s}

\newcommand{\dnsf}{\delta^2_{s}}
\newcommand{\dcf}{\delta_c^2}
\newcommand{\Dcf}{D_c^2}
\newcommand{\Dsf}{D_s^2}

\newcommand{\ns}{\nabla_s}
\newcommand{\sh}{\nabla^2_s}

\newcommand{\h}{\mathbf{H}}
\newcommand{\hba}{\overline{\mathbf{H}}}
\newcommand{\hhat}{\widehat{\mathbf{H}}}
\newcommand{\zero}{\mathbf{0}}
\newcommand{\az}{\alpha_0}
\newcommand{\al}{\alpha}
\newcommand{\abar}{\overline{\alpha}}
\newcommand{\oneh}{\frac{1}{2}}
\newcommand{\Y}{\mathcal{Y}}

\newcommand\scalemath[2]{\scalebox{#1}{\mbox{\ensuremath{\displaystyle #2}}}}

\newcommand{\Projuv}{\Proj_{S, \overline{T}}}
\newcommand{\T}{\overline{T}}
\newcommand{\shess}{\nabla_s^2 f(x^0;S,\T)}

\newcommand{\shessti}{\nabla_s^2 f(x^0;S,\Ti)}

\newcommand{\Ti}{T_{1:m}}

\newcommand{\Projuivi}{\Proj_{S,T_{1:m}}}
\newcommand{\ProjS}{\Proj_S}
\newcommand{\cshtwo}{\nabla^2_{c}}
\newcommand{\cshesstitwo}{\nabla_{c}^2 f(x^0;S,\Ti)}
\newcommand{\cshesstwo}{\nabla_{c}^2 f(x^0;S,\T)}

\begin{document}
\title{A matrix algebra approach to approximate Hessians}
\author{Warren Hare\thanks{Department of Mathematics, University of British Columbia, Okanagan Campus, Kelowna, B.C. V1V 1V7, Canada. Research partially supported by NSERC of Canada Discovery Grant 2018-03865. warren.hare@ubc.ca, ORCID 0000-0002-4240-3903}\and Gabriel Jarry--Bolduc\thanks{Department of Mathematics, University of British Columbia, Okanagan Campus, Kelowna, B.C. V1V 1V7, Canada. Research partially supported by Natural Sciences and Engineering Research Council (NSERC) of Canada Discovery Grant 2018-03865. gabjarry@alumni.ubc.ca }\and Chayne Planiden\thanks{School of Mathematics and Applied Statistics, University of Wollongong, Wollongong, NSW, 2500, Australia. Research supported by University of Wollongong. chayne@uow.edu.au, ORCID 0000-0002-0412-8445}}
\maketitle\author

\begin{abstract}
This work presents a novel matrix-based method for constructing an approximation Hessian using only function evaluations. The method requires less computational power than interpolation-based methods and is easy to implement in matrix-based programming languages such as MATLAB.  As only function evaluations are required, the method is suitable for use in derivative-free algorithms.  

For reasonably structured sample sets, the method is proven to create an order-$1$ accurate approximation of the full Hessian.  Under more specialized structures, the method is proved to yield order-$2$ accuracy.  The undetermined case, where the number of sample points is less than required for full interpolation, is studied and error bounds are developed for the resulting partial Hessians.
\end{abstract}



\section{Introduction}

Derivative-free optimization (DFO) is the study of finding the minimum value and minimizers of a function without using gradients or higher-order derivative information \cite{audet2017derivative}. DFO is gaining in popularity in recent years and is useful in cases where gradients are unavailable, computationally expensive, or simply difficult to obtain \cite{audet2017derivative,conn2009introduction}.  DFO methods have been used in both smooth \cite{berghen2005condor,cocchi2018implicit,Gratton2020,liuzzi2019trust,powell2003trust,powell2009bobyqa,shashaani2018astro,wild2013global} and nonsmooth \cite{audet2018algorithmic,bagirov2008discrete,hare2013derivative,hare2016proximal,larson2016manifold,MMSMW2017} optimization, and most commonly take the form of either direct-search \cite{amaioua2018efficient, audet2008nonsmooth,audet2018mesh,BBN2018,Gratton2017direct} or  model-based \cite{hare2019derivative,liuzzi2019trust,Maggiar2018,verderio2017construction,wild2013global} methods, while some use a blend of both \cite{AudetIanniLeDigTribes2014,MR2457346}.
Model-based DFO methods use numerical analysis techniques to  approximate gradients and Hessians in a manner that has controllable error bounds. For example, a technique to approximate gradients is the  simplex gradient. This method consists of taking a set of $n+1$ sample points in $\R^n$ that are appropriately spaced, called a simplex, and using them to build a linear interpolation function that approximates the objective function locally, then calculating the gradient of the affine function. The simplex gradient has been shown to have a nicely bounded error, for functions whose gradients are locally Lipschitz continuous \cite{bortz1998simplex}. 

One recent line of research explores methods of improving or generalizing the simplex gradient for its use in DFO \cite{conn2008geometry,conn2008bgeometry,MR3935094,custodio2007,MR4074016,hare2020error,powell2001lagrange,regis2015}. Working from ideas in \cite{conn2008bgeometry},  in \cite{MR4074016} the simplex gradient was generalized so as not to require exactly $n+1$ points; an error-controlled approximation can now be found using any finite number of properly-spaced points. In \cite{hare2020error}, a similar approach was used to examine the generalized centered simplex gradient, an approximation that uses twice as many points, but results in an improvement on the error control from order-$1$ to order-$2$  \cite{hare2020discussion}. 
%All three papers \cite{MR4074016,hare2020error,MR3348587} include a study of calculus rules as they can be applied to the numerical approximations that they examine. 

In a similar vein,  researchers have also explored methods to approximate full Hessians or partial Hessians.  In \cite{custodio2007}, the authors outline an idea for a ``simplex Hessian'' that is constructed via
quadratic interpolation through $(n+1)(n+2)/2$ well-poised sample
points.  They further suggested that if only a portion of the Hessian
were desired (say the diagonal component), then fewer points could be used. These ideas were formalized in \cite{conn2008bgeometry} through quadratic interpolation and analyzed through the use of Lagrange interpolating polynomials. Obtaining an approximation of the diagonal component of a Hessian is also discussed in \cite{coope2021gradient,jarry2022approximating}.

In this work, we continue the development of these tools by introducing the \emph{generalized simplex Hessian (GSH)} and the \emph{generalized centered simplex Hessian (GCSH)}. The GSH is closely linked to the Hessian of a quadratic interpolation function. We will see that in certain situations, both approaches yield the same result. 

 This paper can  be viewed as an extension of the work related to simplex Hessians  introduced in   \cite{conn2009introduction,custodio2007}.   We  introduce an explicit formula based on matrix algebra concepts  to approximate the Hessian of a function $f$.  This formula provides several advantages.  First, the formula provides an approach that is well-defined  regardless  of the number of sample points utilized.  Indeed, as long as the matrices used to build the sample set of points  are nonempty, the GSH and the GCSH are well-defined. In particular, the GSH provides an  explicit formula  to approximate the Hessian even when the quadratic interpolation function  does not exist or is not unique.  However, we note that when the matrices used in the computation of the GSH have a specific structure, the GSH is equivalent to a  \emph{forward-finite-difference approximation} of the Hessian.

The GSH provides an accurate approximation of the Hessian of an objective function $f:\R^n\to\R$ under reasonable assumptions. The technique uses a point of interest $x^0$ and sets of directions, stored in a set of matrices $\{S, T_1,\ldots,T_m\}$ whose columns are vectors that are added to the point of interest to build the sample points.  
%However, these vectors need not be distinct, so if we make careful choices, then $\frac{1}{2}(n+1)(n+2)$ function evaluations are sufficient to compute the GSH.
Defining $\Delta_S$ as the radius of $S$ and $\Delta_T$ as the maximum radius of the $T_i$, we prove that if $S$ and all $T_i$ are full row rank and the Hessian $\nabla^2f$ exists and is Lipschitz continuous, then the GSH is an accurate estimate of $\nabla^2f$ to within a multiple of $\Delta_u:=\max\{\Delta_S,\Delta_T\}$.  In terms of order-$N$ accuracy \cite{hare2020discussion}, the GSH provides order-$1$ accurate approximations of the full Hessian (see Definition \ref{df:orderNaccuracy} herein). Furthermore, the GCSH provides  order-2 accurate approximations of the full Hessian. 
Error bounds  for the \emph{underdetermined case} (see Definition  \ref{def:cases}) are presented showing that the GSH and  its centered version, the GCSH, are  order-1  and order-2  accurate approximation of the appropriate partial Hessian, respectively.

The error bounds presented  in this paper  share  similarities to those presented in \cite{conn2008geometry,conn2008bgeometry}. In those papers, the authors provide error bounds for  the Hessian  of a fully quadratic model \cite[Theorem 3]{conn2008geometry}, the Hessian of a quadratic regression model \cite[Theorem 3.2]{conn2008bgeometry}, and the Hessian  of a quadratic undetermined model \cite[Theorem 5.12]{conn2008bgeometry}.  

A second benefit of the GSH formula is that by employing a matrix structure in the definition, the Hessian approximation is extremely easy to implement in any matrix-based programming language.  Moreover, after function evaluations are complete, the matrix structure requires less computational effort (in flops) than quadratic interpolation.  Indeed, when $S$ and all $T_i$ are square matrices, then the GSH requires order $O(n^4)$ flops to compute.  This is an improvement to the order  $O(n^6)$ flops  required to find the Hessian of the quadratic interpolation function using a set of sample points that is poised for quadratic interpolation  \cite{conn2008geometry} (see also \cite[\S 6.2]{conn2009introduction}).

The remainder of this paper is organized as follows. Section \ref{sec:prelim} contains a description of notation and some needed definitions, including those of the generalized simplex gradient, the GSH and the GCSH. Section \ref{sec:relation} clarifies the relation between the GSH and the GCSH. Section \ref{sec:errorbounds} presents several error bounds for the GSH and the GCSH. The error bounds are divided into two categories: the situation where all matrices of directions $T_i$ are equal, and the situation where all matrices $T_i$ are not necessarily equal.  For boths situations, the error bound is valid for the underdetermined case. In Section \ref{sec:qishc}, we establish the minimal poised set for GSH computation and show that it is well-poised for quadratic interpolation. Finally, Section \ref{sec:conc} contains concluding remarks and recommends areas of future research in this vein.

\section{Preliminaries}\label{sec:prelim}

Throughout this work, we use the standard notation found in \cite{rockwets}. The domain of a function $f$ is denoted by $\dom f$. The transpose of a matrix $A$ is denoted by $A^\top$. We work in finite-dimensional space $\R^n$ with inner product $x^\top y=\sum_{i=1}^nx_iy_i$ and induced norm $\|x\|=\sqrt{x^\top x}$. The identity matrix in $\R^{n \times n}$ is denoted by $\Id_n$. We use $e^i_n \in \R^n$ for $i \in \{1,  \dots, n\}$ to denote the standard unit basis vectors in $\R^n$, i.e.\ the $i$\textsuperscript{th} column of $\Id_n.$ The matrix $\Diag(v)= \Diag(v_1, \dots, v_n)=:M \in \R^{n \times n},$ where $v \in \R^n,$  represents the diagonal matrix  with  diagonal entries $M_{i,i}=v_i$ for all $i \in \{1, \dots, n\}.$ The zero vector in $\R^n$ is denoted $\zero$ and the zero matrix in $\R^{n \times m}$ is denoted $\zero_{n \times m}$. The entry in the $i$\textsuperscript{th} row and $j$\textsuperscript{th} column of a matrix $A$ is denoted $A_{i,j}.$ 

Given a matrix $A \in \R^{n \times m},$ we use the induced matrix norm 
\begin{align*}
    \Vert A \Vert=\Vert A \Vert_2=\max \{ \Vert Ax\Vert_2 \, : \, \Vert x \Vert_2=1 \}
\end{align*}
and the Frobenius norm 
\begin{align*}
    \Vert A \Vert_F=\left (\sum_{i=1}^n \sum_{j=1}^m A_{i,j}^2 \right )^{\frac{1}{2}}.
\end{align*}
We denote by $B(x^0,\Delta)$ and $\overline{B}(x^0,\Delta)$ the open and closed balls, respectively, centered about $x^0$ with radius $\Delta$.  The \emph{Minkowski sum} of two sets of vectors $A$ and $B$ is denoted by $A \oplus B.$ That is
\begin{align*}
    A \oplus B=\{a+b: a\in A, b \in B\}.
\end{align*}We define a quadratic function $Q: \R^n \to \R$ to be a function of the form $Q(x)=\alpha_0+\alpha^\top x+\frac{1}{2}x^\top \h x$ where $\alpha_0 \in \R, \alpha \in \R^n$ and $\h=\h^\top \in \R^{n \times n}.$ An affine function $\mathcal{L}:\R^n \to \R$ is defined to be any function that can be written in the form $\mathcal{L}(x)=\alpha_0+\alpha^\top x$. Note, affine functions and constant functions $C(x)=\alpha_0$ are also considered quadratic functions, with $\h=\zero_{n \times n}$. 

Next, we introduce fundamental definitions and notation that will be used throughout this paper.
\begin{df}[Poised for quadratic interpolation]\emph{\cite{conn2008geometry,conn2008bgeometry,conn2009introduction}} \label{def:poisedqi}
The set of distinct points $\Y=\{ y^0, y^1, \dots, y^m \} \subset \R^n$ with $m=\frac{1}{2}(n+1)(n+2)-1$ is poised for quadratic interpolation if the system
\begin{align} 
    \alpha_0+\alpha^\top y^i +\frac{1}{2} (y^i)^\top \h y^i=0, \quad i \in \{0, 1, \dots, m\},  \label{eq:poisedqisys}
\end{align}
has a unique solution for $\alpha_0 \in \R, \alpha \in \R^n,$ and $\h=\h^\top \in \R^{n \times n}.$
\end{df}
When a set of sample points $\Y$ is poised for quadratic interpolation, it means that there exists a unique quadratic model passing through all the sample points in $\Y.$ Therefore, the Hessian of the quadratic model defined through 
 $\Y$ is unique. Note that a set of sample points containing $(n+1)(n+2)/2$ distinct sample points is not necessarily poised for quadratic interpolation. The set of sample points is poised for quadratic interpolation if and only the matrix associated to the linear system \eqref{eq:poisedqisys}  is square and full rank. For instance,  the set of sample points $$\left \{\begin{bmatrix}1\\0\end{bmatrix}, \begin{bmatrix}0\\1\end{bmatrix}, \begin{bmatrix}-1\\0\end{bmatrix}, \begin{bmatrix}0\\-1\end{bmatrix}, \begin{bmatrix}\frac{1}{\sqrt{2}}\\ \frac{1}{\sqrt{2}}\end{bmatrix}, \begin{bmatrix}-\frac{1}{\sqrt{2}}\\ -\frac{1}{\sqrt{2}}\end{bmatrix}   \right \} \subset \R^2$$ is not poised for quadratic interpolation since the rank of the matrix associated to the system \eqref{eq:poisedqisys} is $5<6$. The set $$\left \{\begin{bmatrix}0\\0\end{bmatrix}, \begin{bmatrix}1\\0\end{bmatrix}, \begin{bmatrix}0\\1\end{bmatrix}, \begin{bmatrix}-1\\0\end{bmatrix}, \begin{bmatrix}0\\ -1\end{bmatrix}, \begin{bmatrix}\frac{1}{\sqrt{2}}\\ -\frac{1}{\sqrt{2}}\end{bmatrix}   \right \} \subset \R^2$$ is poised for quadratic interpolation  since the rank of the matrix  associated to the  system  \eqref{eq:poisedqisys} is 6. When the number of distinct sample points in $\Y$ is fewer than $(n+1)(n+2)/2,$ then  the quadratic interpolation model is no longer unique. This case is referred as the underdetermined case \cite[Chapter 5]{conn2009introduction}. In this case,  a quadratic model may be defined through the minimum Frobenius norm problem \cite[Chapter 5]{conn2009introduction} (see also \cite[Section 5]{conn2008bgeometry}).  When the number of distinct sample points in $\Y$ is greater than $(n+1)(n+2)/2$ and  there exists no quadratic function that passes through all points in $\Y,$ the quadratic model  may be determined through a least square regression problem \cite[Chapter 4]{conn2009introduction} (see also \cite[Sections 2,3 \& 4]{conn2008bgeometry}). 
\begin{df}[Quadratic interpolation function]\emph{\cite[Definition 9.9]{audet2017derivative}} \label{def:quadinterpolationfunc}
Let $f:\dom f\subseteq \R^n \to \R$ and let $\Y=\{y^0, y^1, \dots, y^m\} \subset \dom f$ with $m=\frac{1}{2}(n+1)(n+2)-1$ be poised for quadratic interpolation. Then the quadratic interpolation function of $f$ over $\Y$ is
\begin{align*}
    Q_f(\y)(x)&=\alpha_0+\alpha^\top x+\frac{1}{2}x^\top \h x,
\end{align*}
where $(\alpha_0, \alpha, \h=\h^\top)$ is the unique solution to
\begin{align*}
    \alpha_0+\alpha^\top y^i+\frac{1}{2}(y^i)^\top \h y^i&= f(y^i), \quad i \in \{0, 1, \dots, m\}.
\end{align*}
\end{df}

In the next definition, we introduce key notation used in the construction of the GSH.  Within, we write a set of vectors in matrix form, by which we mean that each column of the matrix is a vector in the set.

\begin{df}[GSH notation]
Let $f: \dom f \subseteq \R^n \to \R$ and let $x^0 \in \dom f$ be the \emph{point of interest}.  Let
\begin{align*}
    S&=\begin{bmatrix} s^1&s^2&\cdots&s^m\end{bmatrix} \in \R^{n \times m} ~~\mbox{and}\\
    T_i&=\begin{bmatrix} t^1_i&t^2_i&\cdots&t^k_i \end{bmatrix} \in \R^{n \times k_i}, i\in\{1,\ldots,m\}
\end{align*}
be sets of directions contained in $\R^n$, written in matrix form.  Define
    $$\Ti=\{T_1,\ldots,T_m\}.$$
Assume that $x^0 \oplus T_i, x^0+s^i, x^0+s^i \oplus T_i$ are contained in $\dom f$ for all  $i~\in\{1,\ldots,m\}$. Define 
    $$\Delta_S=\max\limits_{i \,\in\{1,\ldots,m\}}\|s^i\|, \quad \Delta_{T_i}=\max\limits_{j \,\in\{1,\ldots,k_i\}}\|t^j_i\|, \quad \Delta_T=\max\limits_{i\in\{1,\ldots,m\}}\Delta_{T_i},$$
 \begin{align} \label{eq:hat}
\widehat{S}=\frac{1}{\Delta_S} S,
 \end{align}
and
\begin{align*}
\dsf\left (x^0;T_i\right )&=\left[\begin{array}{c}f(x^0+t^1_i)-f(x^0)\\f(x^0+t^2_i)-f(x^0) \\\vdots\\f(x^0+t^k_i)-f(x^0)\end{array}\right]\in\R^{k_i}.
\end{align*}
\end{df}

In this paper, all matrices involved are assumed to  have non-null rank. This ensures that the radius of the matrix is nonzero and hence, \eqref{eq:hat} is always well-defined.
Note that in the previous definition, $m\in\N$ and $k_i$ can be any positive integer for all $i \in \{1, \dots, m\}$.

Recall that for nonsquare matrices, a generalization of the matrix inverse is the pseudoinverse. The most well-known type of matrix pseudoinverse is the Moore--Penrose pseudoinverse.

\begin{df}[Moore--Penrose pseudoinverse] \label{def:mpinverse}
Let $A\in\R^{n\times m}$. The \emph{Moore--Penrose pseudoinverse} of $A$, denoted by $A^\dagger$, is the unique matrix in $\R^{m \times n}$ that satisfies the following four equations:
\begin{enumerate}[(i)]
\item $AA^\dagger A=A,$ 
\item $A^\dagger AA^\dagger=A^\dagger,$ \label{eq:mpi2}
\item $(AA^\dagger)^\top=AA^\dagger,$
\item $(A^\dagger A)^\top=A^\dagger A.$
\end{enumerate}
\end{df}

Note that given $A \in \R^{n \times m},$ there exists a unique Moore--Penrose pseudoinverse $A^\dagger\in\R^{m\times n}.$ The following two properties hold.
\begin{enumerate}[(i)]
\item If $A$ has full column rank $m$, then $A^\dagger$ is a left-inverse of $A$, that is $A^\dagger A=\Id_m.$ In this case, 
\begin{equation}
A^\dagger=(A^\top A)^{-1} A^\top. \label{eq:fullrowrank}
\end{equation}
\item If $A$ has full row rank $n$, then $A^\dagger$ is a right-inverse of $A$, that is $AA^\dagger=\Id_n.$ In this case, $A^\dagger=A^\top (A A^\top)^{-1}.$
\end{enumerate}

Before introducing the GSH, it is valuable to recall the definition of the generalized simplex gradient. 

\begin{df}[Generalized simplex gradient] \emph{\cite[Definition 2]{MR4074016}} Let $f:\dom f \subseteq \R^n\to\R$ and let $x^0 \in \dom f$ be the point of interest.  Let $S\in \R^{n \times m}$ with $x^0 \oplus S \subset \dom f.$ The \emph{generalized simplex gradient} of $f$ at $x^0$ over $S$ is denoted by $\ns f(x^0;S)$ and defined by
\begin{equation*}\label{eq:simplex}\ns f(x^0;S)=(S^\top)^\dagger\ds (x^0;S).\end{equation*}
\end{df}
The Moore--Penrose pseudoinverse allows for $S$ to contain any finite number of columns, rather than being restricted to $m=n$. When $S\in\R^{n \times n}$ and $\rank S=n$, then the Moore--Penrose pseudoinverse of $S^\top$ is the inverse of $S^\top$ and we recover the definition of the simplex gradient \cite[Definition 9.5]{audet2017derivative}. Note that the generalized simplex gradient is a generalization of the familiar \emph{forward-finite-difference approximation} of the gradient \cite[Section 4.1]{Burden2016}. Indeed, if $S=H\Id_n,$ where $H=\Diag \begin{bmatrix} h_1&h_2&\dots&h_n\end{bmatrix}, h_i>0,$ then we recover the forward-finite-difference approximation of the gradient.

We now introduce the two key definitions for this paper, the generalized simplex Hessian and the  generalized centered simplex Hessian.   It requires finite sets of directions and a point of interest at which the Hessian is approximated.  Similarly to the generalized simplex gradient, the GSH involves the Moore--Penrose pseudoinverse and a difference matrix.  In this case, the difference matrix  consists of the difference between generalized simplex gradients.

For ease of notation, we use the definition $\Ti=\{T_1,\ldots,T_m\}$ where convenient.

\begin{df}[Generalized simplex Hessian]\label{def:gsh}
Let $f:\dom f \subseteq \R^n \to \R$ and let $x^0 \in \dom f$ be the point of interest.  Let $S=\begin{bmatrix} s^1&s^2&\cdots &s^m \end{bmatrix}  \in \R^{n \times m}$ and $ T_i \in \R^{n \times k_i}$ with $x^0 \oplus T_i,x^0 \oplus S,x^0+s^i \oplus T_i$ contained in  $\dom f$ for all $i \in \{1, \dots, m\}.$ The \emph{ generalized simplex Hessian}  of $f$ at $x^0$ over $S$ and $\Ti$ is denoted by $\sh f(x^0;S,\Ti)$ and defined by
\begin{equation}\sh f(x^0;S,\Ti)=(S^\top)^\dagger\dnsf (x^0;S,\Ti), \label{eq:formulaGSH} \end{equation} 
where
\begin{equation*}
    \dnsf(x^0;S;\Ti)=\left[\begin{array}{c}(\ns f(x^0+s^1;T_1)-\ns f(x^0;T_1))^\top\\(\ns f(x^0+s^2;T_2)-\ns f(x^0;T_2))^\top\\\vdots\\(\ns f(x^0+s^m;T_m)-\ns f(x^0;T_m))^\top\end{array}\right]\in \R^{m \times n}.
\end{equation*}
In the case $T_1=T_2=\cdots=T_m$, we use $\overline{T}:=T_i$ to simplify notation and write $\sh f(x^0;S;\T)$ to emphasize the special case.
\end{df}

We remark that the $s$ in $\sh f, \delta_s$ and $\dnsf$ is for simplex.  This becomes important as we now introduce the generalized centered simplex Hessian.  The `square' in  $\dnsf$ emphasizes the fact that we are taking differences of first-order objects (hence getting a second-order object). 

\begin{df}[Generalized centered simplex Hessian] \label{def:gcsh}
 Let $f:\dom f \subseteq \R^n \to \R$ and let $x^0 \in \dom f$ be the point of interest.  Let $S  \in \R^{n \times m}$ and $ T_i \in \R^{n \times k_i}$ with $x^0 \oplus  S  \oplus  T_i,  x^0 \oplus (-S) \oplus (-T_i),  x^0 \oplus (\pm S),$ and $x^0 \oplus (\pm T_i)$ contained   in  $\dom f$ for all $i \in \{1, \dots, m\}.$  The  generalized centered simplex Hessian  of $f$ at $x^0$ over $S$ and $\Ti$ is denoted by $\cshtwo f(x^0;S,T_{1:m})$ and defined by   
\begin{equation}\label{eq:centered2simplexHessian}
\cshtwo f(x^0;S,T_{1:m})=\frac{1}{2} \left ( \shessti + \nabla_s^2 f(x^0;-S,-\Ti) \right ).
\end{equation}
\end{df}

Note that the $c$ in $\cshtwo f$ is for centered.

The GCSH can be viewed as a generalization of the centered-finite difference approximation of the Hessian, also called midpoint approximation \cite[Section 4.1]{Burden2016}. Similarly, The GSH can be  viewed as a generalization of the forward-finite-difference approximation of the Hessian \cite[Section 4.1]{Burden2016}. Indeed, when $S=T_1=\dots=T_m=H\Id_n,$ where $H=\Diag \begin{bmatrix} h_1&h_2&\dots&h_n \end{bmatrix}, h_i>0,$ then the GSH (GCSH) is equivalent to a forward-finite-difference approximation (centered-finite-difference approximation) of the second-order derivatives. Moreover, the GSH can be viewed as the DFO analog of the Hessian approximation given by Taylor's Theorem in the derivative-based context, which approximates the gradient by differences of true gradients  (see \cite[Section 8.1]{nocedal2006numerical}).  An advantage of the GSH (GCSH) over forward-finite-difference approximations (centered-finite-difference approximations) is that it is less restrictive.  In particular, it is well-defined as long as the matrices $S$ and $T_1,\dots,T_m$ are nonempty.

\begin{remark} \label{rem:remark} Consider the case where $S$ and $T_i$ are square matrices, i.e., $m=n$ and $k_i=n$ for all $i \in \{1, \dots, m\}.$  The computation of the GSH begins by the construction of $n+1$ generalized simplex gradients.   Setting aside function evaluations, computing a simplex gradient is dominated by computing a matrix inverse, which is $O(n^3)$ flops.  Thus the construction of $\dnsf(x^0;S;\Ti)$ requires $O(n^4)$ flops.  The inverse of $S$ is computed ($O(n^3)$ flops) and multiplied by $\dnsf(x^0;S;\Ti)$, $O(n^3)$ flops.  Hence, setting aside function evaluations, computing the GSH requires $O(n^4)$ flops.  

Conversely, setting aside function evaluations, quadratic interpolation requires $O(n^6)$ flops to compute.
\end{remark}

\begin{df}\label{def:cases}
Let $f:\dom f \subseteq \R^n \to \R$ and let $x^0 \in \dom f$ be the point of interest.  Let $S=\begin{bmatrix} s^1&s^2&\cdots &s^m \end{bmatrix}  \in \R^{n \times m}$ and $ T_i \in \R^{n \times k_i}$  with $x^0 \oplus  S  \oplus  T_i,  x^0 \oplus (-S) \oplus (-T_i),  x^0 \oplus (\pm S),$ and $x^0 \oplus (\pm T_i)$ contained   in  $\dom f$ for all $i \in \{1, \dots, m\}.$ Assume that all matrices are non-null rank.
We define  the following  four cases to characterize  the matrix $S \in \R^{n \times m}$ and the set of matrices $\Ti:=\{T_1, \dots, T_m\}.$
 \begin{itemize}
\item Underdetermined: the GSH (GCSH) is said to be $S$-underdetermined if $S$ is non-square and full column rank. 
We say that it is $\Ti$-underdetermined if all matrices in the set $T_{1:m}$ are full column rank and at least one matrix is non-square.
\item Determined: the GSH (GCSH) is said to be $S$-determined if $S$ is square and full rank. It is $\Ti$-derdetermined if all matrices in the set $T_{1:m}$ are square and full rank.
\item Overdetermined:  the GSH (GCSH) is said to be $S$-overdetermined if $S$ is non-square and full row rank. 
It is $\Ti$-overderdetermined if all matrices  in the set $T_{1:m}$ are  full row rank  at least one is  non-square.
\item Nondetermined: the GSH (GCSH) is said to be $S$-nondetermined if it is not in any of the previous three cases.  It is $\Ti$-nonderdetermined if the set $\Ti$ is not in any of the previous three cases.
 \end{itemize}
\end{df}
  Note that the definition of an $S$-underdetermined GSH (GCSH) implies that $\Span S\neq\R^n$, which is true if and only if $SS^\dagger \neq \Id_n$.
Similarly, the definition of a $\Ti$-underdetermined GSH (GCSH) implies that $\Span T_i \neq \R^n \text{ for some } \bar i \in \{1, \dots, m\}$, which is true if and only if $T_iT_i^\dagger \neq \Id_n \text{ for }\bar i$.\par

Defining underdetermined, determined, overdetermined and nondetermined as above creates 16 different cases to classify a GSH (GCSH), all of which are investigated in Section \ref{sec:errorbounds}. %In the special case where all matrices $T_i$ are equal, we  may write $\T$- instead of $\Ti$-.

The following proposition presents a bound for the absolute error between the generalized simplex gradient $\nabla_s f(x^0;S)$ and the true gradient $\nabla f(x^0).$  Note that the assumption that $f \in \mathcal{C}^2$ on $B(x^0, \overline{\Delta})$ implies that $\nabla f$ is Lipschitz continuous on $\overline{B}(x^0, \Delta)$ where $\Delta<\overline{\Delta}.$

We remark that results akin to Proposition \ref{prop:GSGerror} have been
proven several times using a variety of different techniques
\cite{custodio2007,regis2015, MR4074016}.  In \cite{conn2008geometry,conn2008bgeometry} (see also
\cite{conn2009introduction} and \cite[Theorem 3.1]{custodio2007}), simplex gradients were
analyzed using Lagrange interpolating polynomials.  In \cite{regis2015},
(see also \cite{audet2017derivative,MR4074016}), error bounds for simplex gradients were
analyzed using matrix condition numbers. In Proposition
\ref{prop:GSGerror}, we present the error bound in this second format, as further results will examine generalized simplex Hessians in the same style.

\begin{prop}[Error bound for the generalized simplex gradient]\label{prop:GSGerror}
 Let $S=\begin{bmatrix} s^1&s^2&\cdots&s^m\end{bmatrix} \in \R^{n \times m}$  with radius $\Delta_S>0$. Let $\widehat{S}=S/\Delta_S$  and let $f:\dom f \subseteq \R^n\to\R$ be $\mathcal{C}^2$ on $B(x^0,\overline{\Delta})$  where $x^0$ is the point of interest and $\overline{\Delta}>\Delta_S.$ Denote by $L_{\nabla f} \geq 0$ the Lipschitz constant of $\nabla f$ on $\overline{B}(x^0,\Delta_S).$ 
 Then
\begin{align*}
\| \ProjS \nabla_s  f(x^0;S)- \ProjS \nabla f(x^0)\| &= \| \nabla_s  f(x^0;S)- \ProjS \nabla f(x^0) \| \\
&\leq \frac{\sqrt{m}}{2}L_{\nabla f}  \left \Vert (\widehat S^\top)^\dagger \right \Vert \Delta_S,
\end{align*}
where $\ProjS (\cdot)=(S^\top)^\dagger S^\top (\cdot).$
\end{prop}

The error bound presented in Proposition \ref{prop:GSGerror} only requires that  $S$  has non-null rank. Therefore, it covers all possible cases: $S$-underdetermined, $S$-determined case, $S$-overdetermined  and $S$-nondetermined. In Section \ref{sec:errorbounds}, error bounds  with a similar format than the one in Proposition \ref{prop:GSGerror} are developed for the GSH and  the GCSH.  To do this, we next define the projection onto a pair of subspaces.

Given matrices $S \in \R^{n \times m}$ and  $T_i\in \R^{n \times k_i}$, the projection of the matrix $H \in\R^{n \times n}$ onto $S$  and $\Ti$ is denoted by $\Projuivi H$ and defined by
\begin{align*}
\Projuivi H&=\sum_{i=1}^m (S^\top)^\dagger e^i_m (e^i_m)^\top S^\top H T_i T_i^\dagger.
\end{align*}
In the case where $T_1=T_2=\dots=T_m=:\T,$ the projection of $H$ onto $S$ and $\T$ is denoted by $\Projuv H,$ and reduces to 
\begin{align*}
\Projuv H&=\sum_{i=1}^m (S^\top)^\dagger e^i_m (e^i_m)^\top S^\top H \T \T^\dagger \\
&=(S^\top)^\dagger \left ( \sum_{i=1}^m  e^i_m (e^i_m)^\top \right )S^\top H\T\T^\dagger \\
&=(S^\top)^\dagger {\Id}_m S^\top H\T\T^\dagger \\
&=(S^\top)^\dagger S^\top H\T\T^\dagger.
\end{align*}

We conclude this section by  defining  \emph{order-N accuracy} as it used in this paper several times to describe the quality of the Hessian  approximation techniques developed.

\begin{df}[Order-$N$ accuracy] \label{df:orderNaccuracy}
Let $P \in \R^{n \times n}$. Given $f \in \mathcal{C}^2$, $x^0 \in \dom f$ and $\overline{\Delta}>0,$ we say that $\{\tilde f_\Delta\}_{\Delta \in (0,\overline{\Delta}]}$ is a class of models of $f$ at $x^0$  parametrized by $\Delta$ that provides order-N accuracy of $P$ at $x^0$ if there exists a scalar $\kappa(x^0)$ such that, given any $\Delta \in (0, \overline{\Delta}],$ the model $\tilde{f}_\Delta$ satisfies
$$\Vert P-\nabla^2 \tilde{f}_\Delta (x^0)\Vert\leq \kappa(x^0)\Delta^N.$$
\end{df}

When defining error bounds in Section \ref{sec:errorbounds},  the matrix $P$ in the previous definition will be equal to $\Projuivi \nabla^2 f(x^0)$.  We may refer to $\Projuivi \nabla^2 f(x^0)$ as a \emph{partial Hessian}.  Note that when  $S$ is full row rank and $T_i$ is full row rank for all $i,$ then  $\Projuivi \nabla^2 f(x^0)$ is equal to the true Hessian $\nabla^2 f(x^0).$   We will see that we can obtain order-1 accuracy, or order-2 accuracy, of the partial Hessian $\Projuivi \shessti$   under the appropriate assumptions.  First, the relation between the GSH and the GCSH is clarified in the next section.

\section{The relation between the GSH and the GCSH} \label{sec:relation}

In this section, we investigate the relation between the GSH and the GCSH. We begin by providing a formula to compute the GCSH in a format  similar to the formula for the GSH \eqref{eq:formulaGSH}. 

In Lemma \ref{lem:gcshrelation1} we introduce the notation $\dcf$.  This is a centered  version of $\dnsf$; hence, the $s$ for simplex is replaced with $c$ for centered. 

\begin{lem}\label{lem:gcshrelation1}
 Let $f:\dom f \subseteq \R^n \to \R$ and let $x^0 \in \dom f$ be the point of interest.  Let $S=\begin{bmatrix} s^1&\cdots&s^m \end{bmatrix}  \in \R^{n \times m}$ and $ T_i \in \R^{n \times k_i}$ with $x^0 \oplus  S  \oplus  T_i,  x^0 \oplus (-S) \oplus (-T_i), x^0 \oplus (\pm S), x^0 \oplus (\pm T_i)$ contained   in  $\dom f$ for all $i \in \{1, \dots, m\}.$ Then 
 \begin{align*}
\cshesstitwo&=(S^\top)^\dagger \dcf (x^0;S,\Ti)
 \end{align*}
 where
\begin{equation*}
    \dcf(x^0;S,\Ti)=\frac{1}{2}\begin{bmatrix} \left ( \dsf(x^0+s^1;T_1)+\dsf(x^0-s^1;-T_1)-\dsf(x^0;T_1)-\dsf(x^0;-T_1)  \right )^\top T_1^\dagger \\ \vdots \\  \left ( \dsf(x^0+s^m;T_m)+\dsf(x^0-s^m;-T_m)-\dsf(x^0;T_m)-\dsf(x^0;-T_m)  \right )^\top T_m^\dagger\end{bmatrix}.
\end{equation*}
\end{lem}
\begin{proof}
We have
\begin{align*}
&\cshesstitwo \\
=&\frac{1}{2} \left ( \shessti +\nabla_s^2 f(x^0;-S,-T_{1:m}) \right ) \\
=&\frac{1}{2} (S^\top)^\dagger\left (  \dnsf (x^0;S,T_{1;m})- \dnsf (x^0;-S,-T_{1:m})  \right ) \\
=&(S^\top)^\dagger \frac{1}{2} \begin{bmatrix} \left ( \nabla_s f(x^0+s^1;T_1)-\nabla_s f(x^0;T_1) \right )^\top- \left ( \nabla_s f(x^0-s^1;-T_1)-\nabla_s f(x^0;-T_1) \right )^\top \\ \vdots \\ \left ( \nabla_s f(x^0+s^m;T_m)-\nabla_s f(x^0;T_m) \right )^\top- \left ( \nabla_s f(x^0-s^m;-T_m)-\nabla_s f(x^0;-T_m) \right )^\top \end{bmatrix} \\
=&(S^\top)^\dagger \frac{1}{2} \begin{bmatrix} \left ( \dsf(x^0+s^1;T_1)-\dsf(x^0;T_1)\right )^\top T_1^\dagger -\left ( \left (\dsf(x^0-s^1;-T_1)-\dsf(x^0;-T_1) \right )^\top (-T_1)^\dagger\right ) \\ \vdots \\ \left ( \dsf(x^0+s^m;T_m)-\dsf(x^0;T_m)\right )^\top T_m^\dagger -\left ( \left (\dsf(x^0-s^m;-T_m)-\dsf(x^0;-T_m) \right )^\top (-T_m)^\dagger\right ) \end{bmatrix} \\
=&(S^\top)^\dagger \frac{1}{2}\begin{bmatrix} \left ( \dsf(x^0+s^1;T_1)+\dsf(x^0-s^1;-T_1)-\dsf(x^0;T_1)-\dsf(x^0;-T_1)  \right )^\top T_1^\dagger \\ \vdots \\  \left ( \dsf(x^0+s^m;T_m)+\dsf(x^0-s^m;-T_m)-\dsf(x^0;T_m)-\dsf(x^0;-T_m)  \right )^\top T_m^\dagger\end{bmatrix} \\
=&(S^\top)^\dagger \dcf(x^0;S,T_{1:m}).
\end{align*} 
\end{proof}
Notice that
\begin{align}
&\left [\dsf(x^0+s^i;T_i)+\dsf(x^0-s^i;-T_i)-\dsf(x^0;T_i)-\dsf(x^0;-T_i) \right ]_{j_i}\nonumber\\\nonumber
=&\frac12\big(f(x^0+s^i+t_i^{j_i})-f(x^0+s^i)+f(x^0-s^i-t_i^{j_i})-f(x^0-s^i)\\&-f(x^0+t_i^{j_i})+f(x^0)-f(x^0-t_i^{j_i})+f(x^0)\big) \notag \\
=&\frac12\big(f(x^0+s^i+t_i^{j_i})+f(x^0-s^i-t_i^{j_i})\nonumber\\
&-f(x^0+s^i)-f(x^0-s^i)-f(x^0+t_i^{j_i})-f(x^0-t_i^{j_i})+2f(x^0)\big)\label{eq:centered2functionvalues}
\end{align}
for $j_i \in \{1, \dots, k_i\}$ and $i \in \{1, \dots, m\}.$ We will use \eqref{eq:centered2functionvalues} in Theorem \ref{thm:ebcentered2}. 

Next, we show that the GCSH is actually a special case of the GSH where the matrices of directions 
%$S$ and $\Ti$ used  to compute the GSH
have a specific form. This result is similar  to the result in  \cite[Proposition 2.10]{hare2020error}, which showed that the generalized centered gradient is equal to the generalized simplex gradient when the matrix of directions used  to compute the generalized simplex gradient has the form $A=\begin{bmatrix} S&-S\end{bmatrix}$ for some matrix $S \in\R^{n \times m}. $  We shall use the following lemma.

\begin{lem} \label{lem:moorepenroseA}
Let $A=\begin{bmatrix} S&-S\end{bmatrix} \in \R^{n \times 2m}$  for some matrix $S \in \R^{n \times m}.$ Then 
$$A^\dagger=\frac{1}{2} \begin{bmatrix} S^\dagger\\-S^\dagger \end{bmatrix}.$$
\end{lem}
\begin{proof} This follows from basic algebra.
\end{proof}

\begin{prop} \label{prop:relationGCSHGSH}
Let $f: \dom f \subseteq \R^n \to \R$ and let $x^0 \in \dom f$ be the point of interest. Let $A=\begin{bmatrix} a^1&a^2&\cdots&a^{2m} \end{bmatrix}= \begin{bmatrix} S&-S\end{bmatrix} \in \R^{n \times 2m}$ for some $S=\begin{bmatrix} s^1&\cdots&s^m \end{bmatrix} \in \R^{n \times m}$ and let  $B_i=T_i \in \R^{n \times k_i},$ and $B_{m+i}=-T_i$  for all $i \in \{1, \dots,m \}.$ Suppose that  $x^0 \oplus  S  \oplus  T_i,  x^0 \oplus (-S) \oplus (-T_i), x^0 \oplus (\pm S), x^0 \oplus (\pm T_i)$ are contained   in  $\dom f$ for all $i \in \{1, \dots, m\}.$ Then $$\nabla_s^2 f(x^0;A,B_{1:2m})=\cshesstitwo.$$
\end{prop}
\begin{proof}
We have \small
\begin{align*}
    &\nabla_s^2 f(x^0;A,B_{1:2m}) \\
    &=\begin{bmatrix} S^\top\\-S^\top\end{bmatrix}^\dagger \begin{bmatrix} \left (\nabla_s f(x^0+a^1;B_1) -\nabla_s f(x^0;B_1) \right )^\top \\ \vdots \\  \left (\nabla_s f(x^0+a^{2m};B_{2m}) -\nabla_s f(x^0;B_{2m}) \right )^\top \end{bmatrix} \\
    &=\frac{1}{2}\begin{bmatrix} (S^\top)^\dagger&-(S^\top)^\dagger \end{bmatrix} \begin{bmatrix} \left ( \nabla_s f(x^0+a^1;B_1) -\nabla_s f(x^0;B_1) \right )^\top \\ \vdots \\ \left ( \nabla_s f(x^0+a^{2m};B_{2m}) -\nabla_s f(x^0;B_{2m}) \right )^\top \end{bmatrix} \\
    &=\frac{1}{2}(S^\top)^\dagger \begin{bmatrix} \left (\nabla_s f(x^0+s^1;T_1)-\nabla_s f(x^0;T_1) \right )^\top \\ \vdots \\ \left (\nabla_s f(x^0+s^m;T_m)-\nabla_s f(x^0;T_m) \right )^\top \end{bmatrix} +\frac{1}{2}(-S^\top)^\dagger \begin{bmatrix} \left (\nabla_s f(x^0-s^1;-T_1)-\nabla_s f(x^0;-T_1) \right )^\top \\ \vdots \\ \left (\nabla_s f(x^0-s^m;-T_m)-\nabla_s f(x^0;-T_m) \right )^\top \end{bmatrix} \\
   &=\frac{1}{2} \left ( \shessti +\nabla_s^2 f(x^0;-S,-T_{1:m}) \right ) \\ 
   &=\cshesstitwo.
\end{align*}\normalsize
\end{proof}

We are now ready to examine the error bounds for the GSH and the GCSH.

\section{Error bounds}\label{sec:errorbounds}

In this section, we provide error bounds for the GSH and the GCSH. These error bounds are separated in two categories: the general case where all matrices $T_i$ are not necessarily equal, and the special case where all matrices $T_i=:\T$ are equal.   The error bounds for the GSH show an order-1 accurate approximation of the full Hessian $\nabla^2 f(x^0)$, or an order-1 accurate approximation of the partial Hessian $\Projuivi \nabla^2 f(x^0)$ if the  GSH is $S$- or  $\Ti$-underdetermined.  Similarly, the error bounds for the GCSH show an order-2 accurate approximation of the full Hessian $\nabla^2 f(x^0)$, or an order-2 accurate approximation of the partial Hessian $\Projuivi \nabla^2 f(x^0)$ if the  GCSH is $S$- or $\Ti$-underdetermined.

Before introducing  the error bounds, we note  that in certain situations, the projection  of  the GSH (GCSH) onto $S$ and $T_{1:m}$  is equal to the GSH (GCSH). 
 \begin{prop} \label{prop:projectiondoesnotaffect}
Let $f: \dom f \subseteq \R^n \to \R$ and $x^0 \in \dom f.$ Let  $S=\begin{bmatrix} s^1&\cdots&s^m \end{bmatrix} \in \R^{n \times m}$ and $T_i \in \R^{n \times k_i}.$  Assume that $x^0 \oplus (\pm S), x^0 \oplus (\pm T_i), x^0 \oplus S \oplus T_i$ and $x^0 \oplus -S \oplus -T_i$ are contained in $\dom  f$ for all $i.$ Then the following hold.
\begin{enumerate}[(i)]
\item If $S$ is full column rank {\bf or}  $T_i$ is full row rank for all $i \in \{1, \dots, m\}$, then 
\begin{align}\label{eq:partialdiag}
\Projuivi \shessti&=\shessti
\end{align}
%\item Let $A=\begin{bmatrix} S&-S \end{bmatrix} \in \R^{n \times 2m}$ where $S \in \R^{n \times m},$ and $B_i=T_i,$ $B_{m+i}=-T_i$  for all $i \in \{1, \dots, m\}.$ If $S$ is full column rank or $T_i$ is full row rank for all $i \in \{1, \dots, m\}$,  then 
and
\begin{align} \label{eq:projectionAsymmetric}
\Projuivi \nabla_c^2 f(x^0;S,T_{1:m})=\nabla_c^2 f(x^0;S,T_{1:m}).
\end{align}
\item If $T_1=T_2=\dots=T_m=:\T,$ then 
\begin{align} \label{eq:allTequal}
\Projuv \shess&=\shess
\end{align}
and
\begin{align}\label{eq:AsymmetricTallthesame}
\Projuv \nabla_c^2 f(x^0;S,\T)&=\nabla_c^2 f(x^0;S,\T).
\end{align}
\end{enumerate}
\end{prop}
\begin{proof}
Suppose $S$ is full column rank. Since $T_i^\dagger T_i T_i^\dagger=T_i^\dagger$ by Property \eqref{eq:mpi2} of the Moore--Penrose pseudoinverse, we obtain
\begin{align*}
\Projuivi \shessti&=\sum_{i=1}^m (S^\top)^\dagger e^i_m (e^i_m)^\top S^\top \shessti T_i T_i^\dagger \\
&=\sum_{i=1}^m (S^\top)^\dagger e^i_m (e^i_m)^\top S^\top (S^\top)^\dagger \dnsf(x^0;S,T_{1:m}) T_iT_i^\dagger \\
&=(S^\top)^\dagger \sum_{i=1}^m  \Diag(e^i_m)  \dnsf(x^0;S,T_{1:m}) T_iT_i^\dagger \\
%&=(S^\top)^\dagger \sum_{i=1}^{m} \begin{bmatrix} \zero_n^\top \\ \vdots \\ \zero_n^\top \\ (\dsf(x^0+s^i;T_{i})-\dsf(x^0;T_{i})^\top T_i^\dagger \\ \zero_n^\top \\ \vdots \\ \zero_n^\top  \end{bmatrix} T_iT_i^\dagger \\
%&=(S^\top)^\dagger \sum_{i=1}^{m} \begin{bmatrix} \zero_n^\top \\ \vdots \\ \zero_n^\top \\ (\dsf(x^0+s^i;T_{i})-\dsf(x^0;T_{i})^\top T_i^\dagger \\ \zero_n^\top \\ \vdots \\ \zero_n^\top  \end{bmatrix}  \\
&=(S^\top)^\dagger \dnsf(x^0;S,\Ti)\\
&=\shessti.
\end{align*}
Equation \eqref{eq:projectionAsymmetric} follows analogously. 

Now suppose that $T_i$ is full row rank for all $i \in \{1, \dots, m\}.$ Then we know $T_iT_i^\dagger=\Id_n$ for all $i \in \{1, \dots, m\}.$ We obtain
\begin{align*}
\Projuivi \shessti&=\sum_{i=1}^m (S^\top)^\dagger e^i_m (e^i_m)^\top S^\top \shessti T_i T_i^\dagger \\
&=\sum_{i=1}^m (S^\top)^\dagger e^i_m (e^i_m)^\top S^\top \shessti \\
&=(S^\top)^\dagger  \left ( \sum_{i=1}^m e^i_m (e^i_m)^\top \right )  S^\top \shessti \\
&=(S^\top)^\dagger    S^\top (S^\top)^\dagger  \dnsf(x^0;S,T_{1:m}) \\
&= (S^\top)^\dagger  \dnsf(x^0;S,T_{1:m})=\shessti.
\end{align*}
Equation \eqref{eq:projectionAsymmetric} follows analogously. 

Finally, suppose that $T_1=T_2=\dots=T_m=:\T$.  We have
\begin{align*}
    \Projuv \nabla_s^2 f(x^0;S,\T)&=(S^\top)^\dagger S^\top \nabla_s^2 f(x^0;S,\T) \T \T^\dagger \\
    &=(S^\top)^\dagger S^\top (S^\top)^\dagger \dnsf (x^0;S,\T) \T \T^\dagger \\
    &=(S^\top)^\dagger S^\top (S^\top)^\dagger \begin{bmatrix}  \left ( \nabla_s f(x^0+s^1;\T)-\nabla_s f(x^0;\T)\right )^\top \\ \vdots \\ \left ( \nabla_s f(x^0+s^m;\T)-\nabla_s f(x^0;\T)\right )^\top \end{bmatrix} \T \T^\dagger \\
    &=(S^\top)^\dagger S^\top (S^\top)^\dagger \begin{bmatrix}  \left ( (\T^\top)^\dagger\dsf (x^0+s^1;\T)-(\T^\top)^\dagger \dsf (x^0;\T)\right )^\top \\ \vdots \\  \left ( (\T^\top)^\dagger\dsf (x^0+s^m;\T)-(\T^\top)^\dagger \dsf(x^0;\T)\right )^\top \end{bmatrix} \T \T^\dagger \\
    &=(S^\top)^\dagger S^\top (S^\top)^\dagger \begin{bmatrix}  \left ( \dsf(x^0+s^1;\T)- \dsf(x^0;\T)\right )^\top \\ \vdots \\  \left ( \dsf(x^0+s^m;\T)- \dsf(x^0;\T)\right )^\top \end{bmatrix} \T^\dagger \T \T^\dagger \\
    &= (S^\top)^\dagger \begin{bmatrix}  \left ( \dsf(x^0+s^1;\T)- \dsf(x^0;\T)\right )^\top \\ \vdots \\  \left ( \dsf(x^0+s^m;\T)- \dsf(x^0;\T)\right )^\top \end{bmatrix} \T^\dagger
\end{align*}
by Property \eqref{eq:mpi2} of the Moore-Penrose pseudoinverse. Thus, $ \Projuv \nabla_s^2 f(x^0;S,\T)= \nabla_s^2 f(x^0;S,\T).$  Equation \eqref{eq:AsymmetricTallthesame} follows analogously. 

%Let $A=\begin{bmatrix} S&-S \end{bmatrix} \in \R^{n \times 2m}$ where $S \in \R^{n \times m}$  and $B_i=T_i$ and  $B_{m+i}=-T_i$  for all $i \in \{1, \dots, m\}.$  Suppose $S$ is full column rank. Then
%\begin{align*}
%\Projuivi \nabla_s^2 f(x^0;A,B_{1:2m})&=\sum_{i=1}^m (S^\top)^\dagger e^i_m (e^i_m)^\top S^\top \shessAB T_i T_i^\dagger \\
%=&\frac{1}{2} \sum_{i=1}^m (S^\top)^\dagger e^i_m (e_m^i)^\top S^\top  \begin{bmatrix} (S^\top)^\dagger&-(S^\top)^\dagger \end{bmatrix} \dnsf (x^0;A,B_{1:2m})T_iT_i^\dagger \\
%=&\frac{1}{2} \sum_{i=1}^m (S^\top)^\dagger e^i_m (e_m^i)^\top S^\top \left ( (S^\top)^\dagger \dnsf(x^0;S,\Ti)T_iT_i^\dagger \right . \\
%&\left . -(S^\top)^\dagger \dnsf (x^0;-S,-\Ti) \right )T_iT_i^\dagger \\
%=&\frac{1}{2} \left ( \left ( \sum_{i=1}^m (S^\top)^\dagger e^i_m (e_m^i)^\top S^\top (S^\top)^\dagger \dnsf(x^0;S,\Ti)T_iT_i^\dagger \right ) \right . \\
%&+  \left .\left ( \sum_{i=1}^m (S^\top)^\dagger e^i_m (e_m^i)^\top S^\top (-S^\top)^\dagger \dnsf(x^0;-S,-\Ti)T_iT_i^\dagger \right ) \right ) \\
%=&\frac{1}{2} \left ((S^\top)^\dagger \dnsf(x^0;S,\Ti)-(S^\top)^\dagger \dnsf(x^0;-S,-\Ti) \right ) \\
%=&(A^\top)^\dagger \dnsf (x^0;A,B_{1:2m})=\shessAB.
%\end{align*}

%Now suppose that $T_i$ is full row rank for all $i \in \{1, \dots, m\}.$ The result may be obtained  using a similar process than the one for \eqref{eq:partialdiag} and using Lemma \ref{lem:moorepenroseA}.

%Finally, \eqref{eq:AsymmetricTallthesame} may be proved using a similar process than the one used to show \eqref{eq:allTequal} and by using Lemma \ref{lem:moorepenroseA}.
\end{proof}

In the following theorem and the remainder of this paper, we use the notation
\begin{align*}
\widehat T_i&=T_i/\Delta_{T_i}, \quad i \in \{1, \dots, m\},\\
\Delta_u&=\max \{\Delta_S, \Delta_{T_1}, \dots, \Delta_{T_m}\}, \\ 
\Delta_l&=\min \{\Delta_S, \Delta_{T_1}, \dots, \Delta_{T_m}\}, \\ 
\Delta_T&=\max \{ \Delta_{T_1}, \dots, \Delta_{T_m} \}\\
\widehat{T}&=\widehat T_i \quad \text{such that} \quad \left \Vert \widehat T_i^\dagger \right \Vert \quad \text{is maximal,} \quad i \in \{1, \dots, m\},  \\
k&= \max \{k_1, \dots, k_m\}, \\
H&=\nabla^2 f(x^0). 
\end{align*}

\begin{thm}[Error bounds for the GSH]\label{thm:mainprop}
Let $f:\dom f \subseteq \R^n\to\R$ be $\mathcal{C}^{3}$ on $B(x^0,\overline{\Delta})$ where $x^0 \in \dom f$ is the point of interest and $\overline{\Delta}>0$. Denote by $L_{\nabla^2 f}\geq 0$ the Lipschitz constant of $\nabla^2 f$ on $\overline{B}(x^0,\overline{\Delta})$ respectively. Let $S=\begin{bmatrix} s^1&s^2&\cdots&s^m\end{bmatrix} \in \R^{n \times m}$ and $T_i=\begin{bmatrix} t^1_i&t_i^2&\cdots&t_i^{k_i}\end{bmatrix} \in \R^{n \times k_i}$  for all $i \in \{1, \dots, m\}.$ Assume that $B(x^0,\Delta_{T_i})\subset B(x^0,\overline{\Delta})$ and $B(x^0+s^i,\Delta_{T_i})\subset B(x^0,\overline{\Delta})$  for all $i \in \{1, \dots m\}.$ 
 
 \begin{enumerate}[(i)]
 \item  If $S$ is full column  rank {\bf or} $T_i$ is full row rank for all $i \in \{1, \dots, m\}$, then
\begin{align}
\| \Projuivi \shessti- \Projuivi \nabla^2 f(x^0)\| &=\|  \shessti-  \Projuivi \nabla^2 f(x^0)\| \notag \\
&\leq 4m\sqrt{k} L_{\nabla^2 f}   \Vert (\widehat{S}^\top )^\dagger  \Vert \Vert \widehat{T}^\dagger \Vert   \left ( \frac{\Delta_u}{\Delta_l} \right )^2  \Delta_u.  \label{eq:seconditem}
\end{align}
\item  If $T_1=T_2=\dots=T_m=:\T$, then 
\begin{align}
\| \Projuv \shess - \Projuv \nabla^2 f(x^0)\| &= \|  \shess- \Projuv \nabla^2 f(x^0)\| \notag\\
&\leq  4 \sqrt{mk} L_{\nabla^2 f} \frac{\Delta_u}{\Delta_l}  \left \Vert(\widehat{S}^\top)^\dagger \right \Vert  \left \Vert  \widehat{T}^\dagger \right \Vert \Delta_u. \label{eq:simplexHessAllTequals}
\end{align}
\end{enumerate}
\end{thm}

\begin{proof} The equalities  in Items $(i)$ and $(ii)$ follow from \eqref{eq:partialdiag} and \eqref{eq:allTequal} respectively. We have 
\begin{align}
 &\|  \shessti- \Projuivi H\| \notag \\
 &= \left \Vert  (S^\top)^\dagger \dnsf (x^0;S,T_{1:m})   -\sum_{i=1}^m (S^\top)^\dagger e^i_m (e^i_m)^\top S^\top H T_i T_i^\dagger  \right \Vert \notag \\
&\leq  \frac{1}{\Delta_S} \left \Vert  (\widehat{S}^\top)^\dagger  \right \Vert  \left \Vert   \dnsf (x^0;S,T_{1:m})- \sum_{i=1}^m (s^i)^\top H T_i T_i^\dagger \right \Vert \notag \\
&= \frac{1}{\Delta_S} \left \Vert  (\widehat{S}^\top)^\dagger  \right \Vert \left \Vert \sum_{i=1}^m (e^i_m)^\top  \dnsf (x^0;S,T_{1:m})- (s^i)^\top H T_i T_i^\dagger\right \Vert  \notag\\
 & \leq \sum_{i=1}^m \frac{1}{\Delta_S \Delta_{T_i}} \left \Vert (\widehat{S}^\top)^\dagger \right \Vert \left \Vert \widehat{T_i}^\dagger \right \Vert  \left \Vert \left ( \dsf (x^0+s^i;T_i)-\dsf(x^0;T_i) \right )^\top -(s^i)^\top H T_i  \right \Vert. \label{eq:plugin}
\end{align}
Let $u_i= \left (\dsf (x^0+s^i;T_i)-\dsf(x^0;T_i) \right )^\top \in \R^{k_i}$ for all $\{1, \dots,m\}.$  Next we find a bound for each $\left [u_i-(s^i)^\top H T_i \right ]_{j_i}, j_i \in \{1, \dots, k_i\}.$ Note that
\begin{align*}
   \left  [u_i-(s^i)^\top H T_i \right ]_{j_i}&=f(x^0+s^i+t_i^{j_i})-f(x^0+s^i)-f(x^0+t_i^{j_i})+f(x^0)-(s^i)^\top H t_i^{j_i}.
\end{align*}
Each function value $f(x^0+s^i+t_i^{j_i}), f(x^0+s^i),$ and $f(x^0+t_i^{j_i}),$  can be written as a second-order Taylor expansion about $x^0$ plus a remainder term $R_2(\cdot).$ Hence, we obtain
\begin{align}
& \left \vert f(x^0+s^i+t_i^{j_i})-f(x^0+s^i)-f(x^0+t_i^{j_i}+f(x^0)-(s^i)^\top HT_i^{j_i} \right \vert \notag \\
&\leq \vert R_2(x^0+s^i+t_i^{j_i}) \vert +\vert R_2(x^0+s^i) \vert +\vert R_2(x^0+t_i^{j_i}) \vert \notag \\
& \leq \ \frac{1}{6}L_{\nabla^2 f} \Vert s^i+t_i^{j_i} \Vert^3+\frac{1}{6}L_{\nabla^2 f} \Vert s^i\Vert^3+\frac{1}{6}L_{\nabla^2 f} \Vert t_i^{j_i}\Vert^3  \notag \\
&\leq \frac{1}{2}L_{\nabla^2 f}(\Delta_S+\Delta_{T_i} )^3. \label{eq:bound1}
\end{align}
Plugging \eqref{eq:bound1} in \eqref{eq:plugin}, we obtain
\begin{align*}
\| \shessti- \Projuivi H\| &\leq \sum_{i=1}^m \frac{1}{\Delta_S \Delta_{T_i}} \left \Vert  (\widehat{S}^\top)^\dagger \right \Vert \left \Vert \widehat{T_i}^\dagger \right \Vert  \frac{\sqrt{k_i}}{2} L_{\nabla^2 f} (\Delta_S+\Delta_{T_i})^3 \\
&\leq \sum_{i=1}^m \frac{1}{\Delta_l^2} \left \Vert  (\widehat{S}^\top)^\dagger \right \Vert  \left \Vert \widehat{T}^\dagger \right \Vert  \frac{\sqrt{k}}{2} L_{\nabla^2 f} (2\Delta_u)^3 \\
&= 4m\sqrt{k} L_{\nabla^2 f} \left \Vert  (\widehat{S}^\top)^\dagger \right \Vert  \left \Vert \widehat{T}^\dagger \right \Vert \left ( \frac{\Delta_u}{\Delta_l} \right )^2  \Delta_u.
\end{align*}
Now, suppose that $T_1=T_2=\dots=T_m=:\T.$  We have
\begin{align}
   \Vert  \nabla_s^2 f(x^0;S,\T)- \Projuv H\Vert&= \left \Vert  (S^\top)^\dagger \dnsf (x^0;S,\T)   - (S^\top)^\dagger  S^\top H \T \T^\dagger  \right \Vert \notag \\
&\leq  \frac{1}{\Delta_S} \left \Vert  (\widehat{S}^\top)^\dagger  \right \Vert  \left \Vert   \dnsf (x^0;S,\T)-  S^\top H \T \T^\dagger  \right \Vert \notag \\
    &\leq \frac{1}{\Delta_S \Delta_T}\Vert(\widehat{S}^\top)^\dagger \Vert \widehat{T}^\dagger \Vert \Vert  \Vert (\Dsf)^\top-S^\top H \T  \Vert, \label{eq:youngstersimplex}
\end{align}
where $$(\Dsf)^\top=\begin{bmatrix} \left ( \dsf (x^0+s^1;\T))-\dsf(x^0;\T)\right )^\top \\ \vdots \\ \left ( \dsf(x^0+s^m;\T)-\dsf(x^0;\T) \right )^\top \end{bmatrix} \in \R^{m \times k}.$$ Each entry  $ \left [(\Dsf)^\top-S^\top H \T \right ]_{i,j}$ has the previous bound obtained in \eqref{eq:bound1}.
Since $$\Vert (\Dsf)^\top -S^\top H \T \Vert \leq \Vert (\Dsf)^\top -S^\top H \T \Vert_F,$$ using \eqref{eq:bound1} in \eqref{eq:youngstersimplex}, we get the inequality
\begin{align*}
    \Vert \nabla_s^2 f(x^0;S,\T)- \Projuv H\Vert &\leq \frac{\sqrt{mk}}{2}  L_{\nabla^2 f} \frac{(\Delta_S+\Delta_T)^3}{\Delta_S \Delta_T} \left \Vert(\widehat{S}^\top)^\dagger \right  \Vert  \left \Vert  \widehat{T}^\dagger \right \Vert \\
     &\leq 4 \sqrt{mk}  L_{\nabla^2 f} \frac{\Delta_u}{\Delta_l}  \left \Vert(\widehat{S}^\top)^\dagger \right \Vert  \left \Vert  \widehat{T}^\dagger \right \Vert \Delta_u.
\end{align*}
\end{proof}

The previous error bounds show that  the GSH is an order-1 accurate approximation of a partial Hessian,  or  an order-1 accurate approximation of the full Hessian when the GSH is $S$-determined\textbackslash overdetermined   and $\Ti$-determined\textbackslash overdetermined (equivalently,  $(S^\top)^\dagger S^\top=\Id_n$ and $T_iT_i^\dagger=\Id_n$ for all $i$). The ratio $\frac{\Delta_u}{\Delta_l}$ suggests that  the radii of the matrices should be taken to be of the same magnitude. If  one of the radius is decreased, it suggests that all radii  should be decreased  by the same amount.  Finally,  the presence of $m$ and $k$ in the error bounds roughly suggests that there is no  advantage of having $S$-overdetermined or $\Ti$-overdetermined. More details on this topic can be found in \cite{braunlimiting,hare2023limiting}.

When all matrices $T_i$ are equal, note that Item $(ii)$ covers  all 16 cases from Definition \ref{def:cases}. 
When all matrices $T_i$ are not equal, Item $(i)$ in the previous theorem covers 12 out of the 16 cases. It does not cover the case $S$-nondetermined\textbackslash overdetermined and $\Ti$-nondetermined \textbackslash underdetermined. These cases are left as a future research direction and discuss in the conclusion. 

\begin{thm}[Error bounds for the GCSH] \label{thm:ebcentered2}
Let $f:\dom f \subseteq \R^n\to\R$ be $\mathcal{C}^{4}$ on $B(x^0,\overline{\Delta})$ where $x^0 \in \dom f$ is the point of interest and $\overline{\Delta}>0$. Denote by $L_{\nabla^3 f}$ the Lipschitz constant of $\nabla^3 f$ on $\overline{B}(x^0,\overline{\Delta}).$ Let $S=\begin{bmatrix} s^1&s^2&\cdots&s^m\end{bmatrix} \in \R^{n \times m},$ $T_i=\begin{bmatrix} t_i^1&t_i^2&\cdots&t_i^{k_i} \end{bmatrix} \in \R^{n \times k_i}$ with $\overline{B}(x^0 \pm s^i,\Delta_{T_i})\subset B(x^0,\overline{\Delta})$  for all $i \in \{1, \dots, m\}$.  
\begin{enumerate}[(i)]
\item If $S$ is full column rank {\bf or} $T_i$ is full row rank for all $i \in \{1, \dots, m\}$, then 
\begin{align}\label{eq:undertildecshesstwo}
\Vert \Projuivi \cshesstitwo - \Projuivi \nabla^2 f(x^0) \Vert\nonumber &= \Vert  \cshesstitwo - \Projuivi \nabla^2 f(x^0) \Vert \notag \\
\leq& 2m  \sqrt{k}L_{\nabla^3 f} \left (\frac{\Delta_u}{\Delta_l} \right )^2  \left \Vert  (\widehat{S}^\top)^\dagger  \right \Vert \left \Vert \left (\widehat{T} \right )^\dagger \right \Vert \Delta^2_u.
\end{align}
 \item If  $T_1=T_2=\dots=T_m=:\T \in \R^{n \times k},$ then
\begin{align}
\left \|\Projuv \cshesstwo- \Projuv \nabla^2 f(x^0) \right \| &=\left \|  \cshesstwo- \Projuv \nabla^2 f(x^0) \right \| \notag \\
&\leq 2  \sqrt{mk}L_{\nabla^3 f} \frac{\Delta_u}{\Delta_l} \left \Vert  (\widehat{S}^\top)^\dagger  \right \Vert \left \Vert \left (\widehat{T} \right )^\dagger \right \Vert \Delta^2_u. \label{eq:tildecentered2ebTiequaltwo} 
\end{align}
\end{enumerate}
\end{thm}
\begin{proof}  The equalities in Items $(i)$ and $(ii)$  follow from \eqref{eq:projectionAsymmetric} and \eqref{eq:AsymmetricTallthesame}  respectively.  We have 
\begin{align}
&\left \Vert \nabla^2_{c} f(x^0;S,\Ti) - \Projuivi H \right \Vert \notag \\
&= \left \Vert  (S^\top)^\dagger \dcf (x^0;S,T_{1:m})   -\sum_{i=1}^m (S^\top)^\dagger e^i_m (e^i_m)^\top S^\top H T_i T_i^\dagger  \right \Vert \notag \\
&\leq  \frac{1}{\Delta_S} \left \Vert  (\widehat{S}^\top)^\dagger  \right \Vert  \left \Vert   \dcf (x^0;S,T_{1:m})- \sum_{i=1}^m (s^i)^\top H T_i T_i^\dagger \right \Vert \notag \\
&= \frac{1}{\Delta_S} \left \Vert  (\widehat{S}^\top)^\dagger  \right \Vert \left \Vert \sum_{i=1}^m (e^i_m)^\top  \dcf (x^0;S,T_{1:m})- (s^i)^\top H T_i T_i^\dagger\right \Vert  \notag\\
   &\leq  \medmath{\sum_{i=1}^m \frac{1}{\Delta_S \Delta_{T_i}} \left \Vert  (\widehat{S}^\top)^\dagger  \right \Vert \left \Vert \widehat{T_i}^\dagger \right  \Vert   \left \Vert \frac{1}{2} \left ( \dsf(x^0+s^i;T_i)+\dsf(x^0-s^i;-T_i)-\dsf(x^0;T_i)-\delta_f^s(x^0;-T_i) \right )^\top   -(s^i)^\top H T_i  \right \Vert.} \label{eq:firstinequality}
\end{align}
Let $w_i=\frac{1}{2} \left ( \dsf (x^0+s^i;T_i)+\dsf (x^0-s^i;-T_i)-\dsf (x^0;T_i)-\dsf (x^0;-T_i) \right )^\top \in \R^{k_i}.$  Next we find a bound for  each entry   $\left [w_i-(s^i)^\top H T_i \right ]_{j_i}, j_i \in \{1, \dots, k_i\}.$ Using \eqref{eq:centered2functionvalues}, we know that\small
\begin{align}
   &\left  [w_i-(s^i)^\top H T_i \right ]_{j_i} \notag \\ 
   =&\frac{1}{2} \left (f(x^0+s^i+t_i^{j_i})+f(x^0-s^i-t_i^{j_i})-f(x^0+s^i)-f(x^0-s^i)-f(x^0+t_i^{j_i})-f(x^0-t_i^{j_i})+2f(x^0) \right )\nonumber\\
   &-(s^i)^\top H t_i^{j_i}. \label{eq:centered2bound}
\end{align}\normalsize
Each function value $f(x^0+s^i+t_i^{j_i}), f(x^0-s^i-t_i^{j_i}), f(x^0+s^i), f(x^0-s^i), f(x^0+t_i^{j_i})$ and $f(x^0-t_i^{j_i})$ can be written as a third-order Taylor expansion about $x^0$ plus a remainder term $R_3(\cdot).$ It follows that
\begin{align}
     & \medmath{\left \vert \frac{1}{2} \left (f(x^0+s^i+t_i^{j_i})+f(x^0-s^i-t_i^{j_i})-f(x^0+s^i)-f(x^0-s^i)-f(x^0+t_i^{j_i})-f(x^0-t_i^{j_i})+2f(x^0) \right )-(s^i)^\top H t_i^{j_i} \right \vert} \notag \\
     &\leq \medmath{\frac{1}{2} \left (\Vert R_3(x^0+s^i+t_i^{j_i})\Vert +\Vert R_3(x^0-s^i-t_i^{j_i})\Vert + \Vert R_3(x^0+s^i)\Vert +\Vert R_3(x^0-s^i) \Vert +\Vert R_3(x^0+t_i^{j_i} \Vert +\Vert R_3(x^0-t_i^{j_i})\Vert \right )} \notag \\
     &\leq \medmath{ \frac1{48}  \left (L_{\nabla^3 f} \Vert s^i+t_i^{j_i}\Vert^4+\frac{1}{24} L_{\nabla^3 f} \Vert -s^i-t_i^{j_i} \Vert^4+L_{\nabla^3 f} \Vert -s^i\Vert^4+L_{\nabla^3 f} \Vert s^i \Vert^4 +L_{\nabla^3 f} \Vert t_i^{j_i} \Vert +L_{\nabla^3 f} \Vert -t_i^{j_i} \Vert   \right )}\notag\\
     &\leq \frac{1}{8} L_{\nabla^3 f} (\Delta_{S}+\Delta_{T_i})^4. \label{eq:deltasdeltatcentered2}
\end{align}
Plugging the bound from \eqref{eq:deltasdeltatcentered2} in \eqref{eq:centered2bound}, we get
\begin{align*}
\Vert \cshtwo f(x^0;S,T_{1:m})- \Projuivi H\Vert
&\leq \sum_{i=1}^m \frac{1}{\Delta_S \Delta_{T_i}}\left \Vert  (\widehat{S}^\top)^\dagger  \right \Vert \left \Vert \widehat{T_i}^\dagger \right \Vert   \sqrt{k_i} \frac{1}{8} L_{\nabla^3 f} (\Delta_{S}+\Delta_{T_i})^4  \\
&\leq \sum_{i=1}^m \frac{1}{\Delta_l^2} \left \Vert  (\widehat{S}^\top)^\dagger  \right \Vert \left \Vert \widehat{T}^\dagger \right  \Vert   \sqrt{k} \frac{1}{8} L_{\nabla^3 f} (2\Delta_u)^4  \\
&=2m  \sqrt{k}L_{\nabla^3 f} \left (\frac{\Delta_u}{\Delta_l} \right )^2  \left \Vert  (\widehat{S}^\top)^\dagger  \right \Vert \left \Vert \left (\widehat{T} \right )^\dagger \right \Vert \Delta^2_u.
\end{align*}

Now, suppose that $T_1=T_2=\dots=T_m=:\T.$  We have
\begin{align}
   \Vert  \cshtwo f(x^0;S,\T)- \Projuv H\Vert&= \left \Vert  (S^\top)^\dagger \dcf (x^0;S,\T)   - (S^\top)^\dagger  S^\top H \T \T^\dagger  \right \Vert \notag \\
&\leq  \frac{1}{\Delta_S} \left \Vert  (\widehat{S}^\top)^\dagger  \right \Vert  \left \Vert   \dcf (x^0;S,\T)-  S^\top H \T \T^\dagger  \right \Vert \notag \\
    &\leq \frac{1}{\Delta_S \Delta_T}\Vert(\widehat{S}^\top)^\dagger \Vert \widehat{T}^\dagger \Vert \Vert  \Vert (\Dcf)^\top-S^\top H \T  \Vert, \label{eq:youngstercentred}
\end{align}
where $$(\Dcf)^\top=\frac{1}{2} \begin{bmatrix} \left ( \dsf (x^0+s^1;\T)+\dsf (x^0-s^1;-\T)-\dsf(x^0;\T)-\dsf (x^0;-\T)\right )^\top \\ \vdots \\ \left ( \dsf(x^0+s^m;\T)+\dsf(x^0-s^m;-\T)-\dsf(x^0;\T)-\dsf(x^0;-\T) \right )^\top \end{bmatrix} \in \R^{m \times k}.$$ Each entry  $ \left [(\Dcf)^\top-S^\top H \T \right ]_{i,j}$ is bounded by \eqref{eq:deltasdeltatcentered2}.
Since $$\Vert (\Dcf)^\top -S^\top H \T \Vert \leq \Vert (D_f^c)^\top -S^\top H \T \Vert_F,$$ using \eqref{eq:deltasdeltatcentered2} in \eqref{eq:youngstercentred}, we get the inequality
\begin{align*}
    \Vert \cshtwo f(x^0;S,\T)- \Projuv H\Vert &\leq \frac{\sqrt{mk}}{24}  L_{\nabla^3 f} \frac{(\Delta_S+\Delta_T)^4}{\Delta_S \Delta_T} \left \Vert(\widehat{S}^\top)^\dagger \right  \Vert  \left \Vert  \widehat{T}^\dagger \right \Vert \\
     &\leq \frac{2 \sqrt{mk} }{3}  L_{\nabla^3 f} \frac{\Delta_u}{\Delta_l}  \left \Vert(\widehat{S}^\top)^\dagger \right \Vert  \left \Vert  \widehat{T}^\dagger \right \Vert \Delta_u^2.
\end{align*}
\end{proof}

The  main differences between the  error bounds presented for the GCSH in Theorem \ref{thm:ebcentered2} and  the GSH in Theorem  \ref{thm:mainprop} are
\begin{itemize}
\item the error bounds in Theorem \ref{thm:ebcentered2} assume that $f \in \mathcal{C}^4$ and the error bounds in Theorem \ref{thm:mainprop} assume $f \in \mathcal{C}^3,$
\item the error bounds in Theorem \ref{thm:ebcentered2} involve the Lipschitz constant $L_{\nabla^3 f}$ and the error bounds in Theorem \ref{thm:mainprop} involve the Lipschitz constant $L_{\nabla^2 f},$
\item the error bounds in Theorem \ref{thm:ebcentered2} are order-2 accurate and the error bounds in Theorem \ref{thm:mainprop} are order-1 accurate.
\end{itemize}

Item $(ii)$ of Theorems \ref{thm:mainprop} and \ref{thm:ebcentered2} is more restrictive than Item $(i)$, but in return it can offer additional benefits, such as symmetry. In  the following proposition, we show that the transpose of the GSH (GCSH) of $f$ at $x^0$ over $S$ and $\overline T$ is the same as the GSH (GCSH) of $f$ at $x^0$ over $\overline T$ and $S$.

\begin{prop} \label{lem:stequalts}
Let $f:\dom f \subseteq \R^n \to \R$ and let $x^0 \in \dom f$ be the point of interest.  Let $S=\begin{bmatrix} s^1&s^2&\cdots &s^m \end{bmatrix}  \in \R^{n \times m}$ and $\overline T=\begin{bmatrix} t^1&t^2&\cdots&t^k \end{bmatrix} \in \R^{n \times k}$  with $x^0 \oplus  S  \oplus  T_i,  x^0 \oplus (-S) \oplus (-T_i),  x^0 \oplus (\pm S),$ and $x^0 \oplus (\pm T_i)$ contained   in  $\dom f$ for all $i \in \{1, \dots, m\}.$Then
$$\left (\nabla^2_s f(x^0;S,\overline T)\right )^\top=\nabla^2_s f(x^0;\overline T,S),$$
and $$\left (\nabla^2_c f(x^0;S,\overline T)\right )^\top=\nabla^2_c f(x^0;\overline T,S).$$
\end{prop}
\begin{proof}
We have
\begin{align*}
    \left (\nabla_s^2 f(x^0;S,\overline T) \right )^\top&=\left ((S^\top)^\dagger \dnsf (x^0;S,\overline T) \right )^\top \\
    &=(\dnsf (x^0;S,\overline T))^\top S^\dagger \\
    &=\begin{bmatrix} \nabla_s f(x^0+s^1;\overline T)-\nabla_s f(x^0;\overline T)&\cdots& \nabla_s f(x^0+s^m;\overline T)-\nabla_s f(x^0;\overline T) \end{bmatrix} S^\dagger \\
    &=(\overline T^\top)^\dagger \begin{bmatrix} \dsf(x^0+s^1;\overline T)-\dsf(x^0;\overline T)&\cdots&\dsf (x^0+s^m;\overline T)- \dsf(x^0;\overline T) \end{bmatrix} S^\dagger. 
\end{align*}
Let $\Dsf=\begin{bmatrix} \dsf(x^0+s^1;\overline T)-\dsf(x^0;\overline T)&\cdots&\dsf (x^0+s^m;\overline{T})- \dsf(x^0;\overline{T}) \end{bmatrix} \in \R^{ k \times m}.$ Then
\begin{align*}
   [\Dsf]_{i,j}&=f(x^0+s^i+t^j)-f(x^0+s^i)-f(x^0+t^j)+f(x^0).
\end{align*}
Note that 
\begin{align*}
    \Dsf&=\begin{bmatrix} \left (\dsf(x^0+t^1;S)-\dsf (x^0;S) \right )^\top \\ \vdots \\ \left (\dsf(x^0+t^k;S)- \dsf (x^0;S) \right )^\top \end{bmatrix}.
\end{align*}
Hence,
\begin{align*}
    \left (\nabla_s^2 f(x^0;S,\overline T) \right )^\top&=(\overline T^\top)^\dagger \begin{bmatrix} \left (\dsf(x^0+t^1;S)-\dsf (x^0;S) \right )^\top \\ \vdots \\ \left (\dsf(x^0+t^k;S)- \dsf (x^0;S) \right )^\top \end{bmatrix} S^\dagger \\
    &=(\overline T^\top)^\dagger \begin{bmatrix} \left ((S^\dagger)^\top \dsf(x^0+t^1;S)- (S^\dagger)^\top \dsf (x^0;S) \right )^\top \\ \vdots \\ \left ((S^\dagger)^\top \dsf(x^0+t^k;S)- (S^\dagger)^\top \dsf (x^0;S) \right )^\top \end{bmatrix} \\
    &=(\overline T^\top)^\dagger \dnsf(x^0;\overline T,S)= \nabla_s^2 f(x^0;\overline T,S). 
\end{align*}
The second equality follows from the definition of the GCSH (Definition \ref{def:gcsh}).  
\end{proof}

In the next section,  we investigate how to obtain an order-$1$ accurate approximation of the full Hessian with a  minimal number of sample points.

\section{Minimal poised set and quadratic interpolation}\label{sec:qishc}
In this section,  we investigate how to choose the matrices of directions $S$ and $T_i$ to obtain an approximation of the full Hessian with a minimal number of sample points. Reducing the number of distinct points  is extremely valuable, as it will decrease the number of distinct function evaluations necessary to compute the GSH.
For all of this section, we set $T_1=\cdots=T_m=:\overline T \in \R^{n \times k}$. We show that if $S$ and $\overline T$ have a specific structure, then the number of distinct function evaluations necessary to compute the GSH is $(n+1)(n+2)/2.$  We then explore some results that occur when such a structure is used. We begin by defining a \emph{set for GSH computation}.

\begin{df}
The set of all distinct points utilized in the computation of $\nabla_s^2 f(x^0;S,\T)$ is said to be the set for {\em GSH computation} and is denoted $\s(x^0;S,\T).$ 
\end{df} 

Note that $\s(x^0;S,\T)$ contains at most  $(m+1)(k+1)$ distinct points, but can contain fewer points if some of them overlap. 
Next, we introduce the  definition of \emph{minimal poised set} for the GSH. 
\begin{df}[Minimal poised set for the GSH]\label{df:minimalGSH}
Let  $x^0 \in \R^n$ be the point of interest.  Let $S=\begin{bmatrix} s^1&s^2&\cdots&s^n\end{bmatrix} \in \R^{n \times n}$ and $\overline T = \begin{bmatrix} t^1 & t^2 & \cdots & t^n \end{bmatrix} \in \R^{n \times n}$.  We say that $\s(x^0;S,\overline T)$ is a \emph{minimal poised set for the GSH} at $x^0$ if and only if $S$ and $\overline T$ are full rank and $\s(x^0;S,\overline T)$ contains exactly $(n+1)(n+2)/2$ distinct points.
\end{df}

This definition requires $S$ and $\overline T$ to have exactly $n$ columns and to be full rank. This is the case where the GSH is  $S$-determined and $\T$-determined. This implies that  $S$ and $\overline T$ are of minimal size to ensure that the GSH is an order-1 accurate approximation of the full Hessian.

Note that  a minimal poised set for the GCSH could be defined in a similar manner. Obviously, such a set must contain more than $(n+1)(n+2)/2.$ An investigation of this topic is left as a future research direction.
We next show that it is possible to create a minimal poised set for the GSH.
\begin{prop} \label{prop:reducingfe}
Let $x^0$ be the point of interest.  Let $S=\begin{bmatrix} s^1&s^2&\cdots&s^n\end{bmatrix} \in \R^{n \times n}$.  Define the set $U_\ell$ for each index $\ell\in\{0,1,\ldots,n\}$ as  $$U_0=S$$ and 
    $$U_\ell=\begin{bmatrix} s^1-s^\ell&s^2-s^\ell&\cdots&s^{\ell-1}-s^\ell&-s^\ell&s^{\ell+1}-s^\ell&\cdots&s^n-s^\ell\end{bmatrix} \in \R^{n \times n},\ell\neq0.$$  Then for each $\ell$, $|\s(x^0;S,U_\ell)| \leq (n+1)(n+2)/2$.  Moreover, if $S$ has full rank, then $|\s(x^0;S,U_\ell)| = (n+1)(n+2)/2$.
\end{prop}
\begin{proof} Without loss of generality, let $x^0=\zero$. First, suppose $\ell\in\{1, 2, \dots, n\}.$  For arbitrary function $f$, consider the matrix $\dnsf (x^0;S,\T)$ where $\T=U_\ell.$ The computation of $\nabla_s f(x^0;U_\ell)$ evaluates $f$ at the points  
    $$\begin{array}{l}
    ~~~\{\zero, (s^1-s^\ell), \dots,  (s^{\ell-1}-s^\ell), -s^\ell,  (s^{\ell+1}-s^\ell),\dots, (s^n-s^\ell)\}\\
    =\{\zero\}\cup\{-s^\ell\}\cup \bigcup\limits_{\substack{i=1\\ i \neq \ell}}^{n} \left\{s^{i}-s^\ell \right\}. \end{array}$$
For $i \neq \ell$, the computation of $\nabla_s f(x^0+s^i;U_k)$ evaluates $f$ at the points 
    $$\begin{array}{l}
    ~~~\{s^i, s^i+(s^1-s^\ell), \dots, s^i+(s^{\ell-1}-s^\ell), s^i-s^\ell,  s^i+(s^{\ell+1}-s^\ell),\dots, s^i+(s^n-s^\ell)\}\\
    =\{s^i\}\cup\{s^i-s^\ell\}\cup \bigcup \limits_{\substack{j=1 \\j \neq \ell}}^{n} \{s^i + s^j - s^\ell\}.
    \end{array}$$
The computation of $\nabla_s f(x^0+s^\ell;U_\ell)$ evaluates $f$ at the points 
   $$\begin{array}{l}
    ~~~\{s^\ell, s^\ell+(s^1-s^\ell), \dots, s^\ell+(s^{\ell-1}-s^\ell), s^\ell-s^\ell,  s^\ell+(s^{\ell+1}-s^\ell),\dots, s^\ell+(s^n-s^\ell)\}\\
    =\{s^\ell\}\cup\{\zero\}\cup \bigcup \limits_{\substack{i=1 \\i \neq \ell}}^{n} \{s^i\} =\{\zero\}\cup S.
    \end{array}$$
Thus, $f$ is evaluated at the points  
    \begin{equation}\label{eq:listofpoints}\begin{array}{l} \tiny{\bigg(\{\zero\} \cup \{-s^\ell\}\cup \bigcup \limits_{\substack{i=1 \\i \neq \ell}}^{n}\left\{s^{i}-s^\ell \right\}\bigg) 
     \cup 
    \bigg(\bigcup \limits_{\substack{i=1\\ i \neq \ell}}^{n} \{s^i\}\cup \bigcup \limits_{\substack{i=1\\i \neq \ell}}^{n} \{s^i-s^\ell\}\cup \bigcup \limits_{\substack{i=1 \\i\neq \ell}}^{n} \bigcup \limits_{\substack{j \geq i\\ j \neq \ell}}\{s^i + s^j - s^\ell\}\bigg)
     \cup 
    \bigg(  \{\zero\}\cup S \bigg)}\\
    = \{\zero\} \cup S \cup\{-s^\ell\}\cup \bigcup \limits_{ \substack {i=1\\i \neq \ell}}^{n} \left\{s^{i}-s^\ell \right\} \cup \bigcup \limits_{\substack{i=1 \\i\neq \ell}}^{n} \bigcup \limits_{\substack{j \geq i\\ j \neq \ell}}\{s^i + s^j - s^\ell\}  .
    \end{array}\end{equation}
This is at most $1+n+1+(n-1)+(n-1)(n-2)/2 = (n+1)(n+2)/2$ points.

Now, suppose $\ell=0.$ Using a similar process to the above, we find that $f$ is evaluated at the points   
\begin{equation}\label{eq:listofpointsu0}\begin{array}{l} \bigg(\{\zero\}\cup \bigcup \limits_{i=1}^n \left\{s^{i} \right \} \bigg) 
    ~ \cup ~
    \bigg( \bigcup \limits_{i=1}^n \{2s^i\}\cup \bigcup \limits_{i =1}^{n} \bigcup \limits_{j>i} \{s^i+s^j\}\bigg).
    \end{array}\end{equation}
This is at most $(n+1)(n+2)/2$ points.

Finally, if $S$ is full rank, then the four sets in  \eqref{eq:listofpoints} and the four sets in \eqref{eq:listofpointsu0} are disjoint, so we have exactly $ (n+1)(n+2)/2$ function evaluations.\qedhere
\end{proof}

Using $S=\Id_n$ in Proposition \ref{prop:reducingfe}, we can create $n+1$ canonical minimal poised sets for the GSH.
\begin{df}[$\ell^{\mbox{th}}$-canonical minimal poised set for the GSH] \label{def:canminset}
Let $x^0 \in \R^n$ be the point of interest.  Let $S=\Id_n$. Fix $\ell \in \{0, 1, \dots, n\}$. Let  
    $$E_0= \Id_n$$
and
    $$E_\ell=\begin{bmatrix} e^1-e^\ell&e^2-e^\ell&\cdots&e^{\ell-1}-e^\ell&-e^\ell&e^{\ell+1}-e^\ell&\cdots&e^n-e^\ell\end{bmatrix}, \ell\neq0.$$ 
Then $\s(x^0;\Id_n,E_\ell)$ is  called the {\em $\ell^{\mbox{th}}$-canonical minimal poised set for the GSH} at $x^0$. 
\end{df}

From Proposition \ref{prop:reducingfe} and the fact that any matrix $E_k$ is full rank in Definition \ref{def:canminset}, the $\ell^{\mbox{th}}$-canonical minimal poised set for the GSH is indeed a minimal poised set for the GSH. Henceforth, we use the notation $\m(x^0;S,U_k)$ to denote a minimal poised set for the GSH at $x^0$ that takes the form constructed in Proposition \ref{prop:reducingfe}.

Note that the order of the directions in $S$ and $\overline T$ is arbitrary.  Thus, it is immediately clear that if $\s(x^0;S,\overline T)$ is a minimal poised set for the GSH and $P_1, P_2 \in \R^{n\times n}$ are permutation matrices, then  $\s(x^0;SP_1,\overline TP_2)$ is also a minimal poised set for the GSH at $x^0$.  The next proposition expands this idea and demonstrates how to construct minimal poised sets for the GSH. 

\begin{prop}\label{prop:invmatrix}
Let $x^0 \in \R^n$ be the point of interest. Let $S,\overline T\in\R^{n \times n}$.  Let $N \in \R^{n \times n}$ be an invertible matrix and $P_1, P_2\in \R^{n \times n}$ be permutation matrices. Then  $\s(x^0;S,\overline T)$ is a minimal poised set for the GSH at $x^0$ if and only if $\s(x^0;NSP_1,N\overline TP_2)$ is a minimal poised set for the GSH at $x^0$.
\end{prop}

\begin{proof} The proof follows trivially from properties of matrices.\end{proof}

It follows that  $\s(x^0;S,\overline T)$ is a minimal poised set for the GSH at $x^0$ if and only if the set $\s(x^0;\beta SP_1, \beta\overline T P_2)$ is a minimal poised set for the GSH at $x^0$, where $\beta$ is a nonzero scalar and $P_1, P_2$ are permutaion matrices.
\begin{ex}
Let $x^0=(0,0)$. The $2^{\mbox{nd}}$-canonical minimal poised set for the GSH in $\R^2$ contains the points $(0,-1),(0,0),(0,1),(1,-1),(1,0)$ and $(2,-1)$.  In this case, $S = \{e^1, e^2\}$ and $\overline T=\{(e^1-e^2), -e^2\}$.  Figure \ref{fig:canset} illustrates this set.  

\begin{figure}[ht]
\caption{The $2^{\mbox{nd}}$-canonical minimal poised set for the GSH at $x^0$ in $\R^2$.} \label{fig:canset}
\includegraphics[scale=0.5]{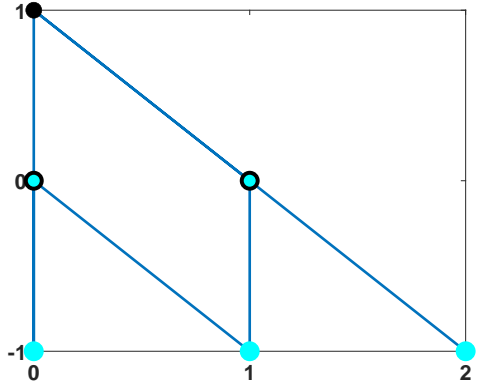}
\centering
\end{figure}

The points in $\{(0,0)\} \cup \{(0,0) \oplus S\} = \{0, e^1, e^2\}$ are represented with solid black borders.  These are the base points where simplex gradients will be computed.  The lines represent the vectors corresponding to $\overline T$ emanate from $\{(0,0)\} \cup \{(0,0) \oplus S\}$.  The points in $\left(\{(0,0)\} \cup \{(0,0) \oplus S\}\right) + \overline T$ are represented with cyan cores.  These are the points used to construct the simplex gradients.   Notice the points $(0,0)$ and $(1,0)$ have both black borders and cyan cores.  These are the common points that allow the number of function evaluations to be reduced to $(n+1)(n+2)/2=6$.
\end{ex}

We next demonstrate that every minimal poised set for the GSH of the form $\m(x^0;S,U_\ell)$ is poised for quadratic interpolation.  We then show that the converse is not true; it is possible to construct a set that is poised for quadratic interpolation, but does not take the form of $\m(x^0;S,U_\ell)$.  

\begin{prop} \label{prop:minimalsetispoisedforqi} Let $S \in \R^{n \times n}$ be full rank.  Select $\ell \in \{0, 1, \ldots, n\}$ and define $\T:=U_\ell$ as in Proposition \ref{prop:reducingfe}.  Then $\m(x^0;S,\T)$ is poised for quadratic interpolation.
\end{prop}
\begin{proof}  Noting that $\m(x^0;S,\T)$ is poised for quadratic interpolation if and only if the set  $\m(x^0;S,\T)\oplus -x^0$ is poised for quadratic interpolation, we assume without loss of generality that $x^0=\zero$.\smallskip\\
\noindent{\bf Case I: $\T=U_\ell$ for $\ell\in \{1, 2, \ldots, n\}$.}  Without loss of generality, by Proposition \ref{prop:invmatrix}, assume that $\ell=n$. 
The points contained in $\m(x^0;S,U_n)$ are
\begin{align*}
    &\{\zero\} \cup S \cup\{-s^n\}\cup \bigcup \limits_{i=1}^{n-1}\left\{s^{i}-s^n \right\} \cup  \bigcup \limits_{i=1}^{n-1} \bigcup \limits_{\substack {j \geq i \\ j \neq n}} \{s^i + s^j - s^n\}\\
    =&\{\zero\} \cup S \cup\{-s^n\}\cup \bigcup \limits_{i=1}^{n-1} \left\{s^{i}-s^n \right\} \cup \{2s^i-s^n\}_{i=1}^{n-1} \cup \bigcup \limits_{i=1}^{n-1} \bigcup \limits_{\substack{ j>i \\j \neq n}} \left\{s^{i}+{s^j}-s^n \right\}.
\end{align*}
We show that using this set, the only solution  to \eqref{eq:poisedqisys} is the trivial solution.  
Considering the point $\zero$, we obtain
\begin{align*}
    \az&=0.
\end{align*}
Considering the points $s^i$ for $i \in \{1, 2, \dots, n\}$ and noting that for all $i \in \{1, 2, \dots, n\},$ $s^i=Se^i$, we obtain
\begin{align}
    \abar^\top e^i+ \oneh (e^i)^\top \hhat e^i &=0, \label{eq:x0h}
\end{align}
where $\abar^\top=\al^\top S$ and $\hhat=S^\top \h S.$ Note that $\hhat$ is symmetric.  
Considering the points $s^i-s^n$ for $i \in \{ 1, 2, \dots, n-1\},$ we obtain
\begin{equation*}
     \abar^\top e^i -\abar^\top e^n -(e^i)^\top \hhat e^n + \oneh (e^i)^\top \hhat e^i + \oneh (e^n)^\top \hhat e^n=0.
\end{equation*}
Using \eqref{eq:x0h}, this simplifies to 
\begin{align}
  -\abar^\top e^n -(e^i)^\top \hhat e^n + \oneh (e^n)^\top \hhat e^n=0. \label{eq:x0eien}
\end{align}
Considering the point $-s^n,$ we find
\begin{align}
    -\abar^\top e^n + \oneh (e^n)^\top \hhat e^n&=0, \label{eq:aennHn}
\end{align}
which reduces \eqref{eq:x0eien} to 
\begin{align*}
  -(e^i)^\top \hhat e^n =0 \quad \mbox{for }~i \in \{1, 2, \ldots, n-1\}.
\end{align*}
Thus, $(e^i)^\top \hhat e^n=\hhat_{i,n}=\hhat_{n,i}=0$ for all $i \in \{1 ,2, \dots, n-1\}.$ Combining \eqref{eq:x0h} at $i=n$ and \eqref{eq:aennHn} multiplied by $-1$, we get
\begin{align*}
    \abar^\top e^n+ \oneh (e^n)^\top \hhat e^n  = 0 = \abar^\top e^n - \oneh (e^n)^\top \hhat e^n,
\end{align*}
which implies that $(e^n)^\top \hhat e^n= \hhat_{n,n}=0$. Considering the points $2s^i-s^n$ for $i \in \{1, 2, \dots, n-1\},$ we get
\begin{equation*}
    2 \abar^\top e^i - \abar^\top e^n +2 (e^i)^\top \hhat e^i + \oneh (e^n)^\top \hhat e^n=0.
\end{equation*}
Using \eqref{eq:aennHn}, this simplifies to
\begin{align*}
    2 \abar^\top e^i+ 2 (e^i)^\top \hhat e^i=0.
\end{align*}
By multiplying \eqref{eq:x0h} by 2 and substituting in the above equation, we get $(e^i)^\top \hhat e^i=\hhat_{i,i}=0$ for all $i \in \{1, 2, \dots, n-1\}.$  This now implies $\abar_i = \abar^\top e^i=0$ for all $i \in \{1, 2, \ldots, n\}$, i.e., $\abar=\zero$. Lastly, consider the points $s^i+s^j-s^n$ for $i \neq j, i,j \in \{ 1,2, \dots, n-1\}.$ Since $\hhat_{i,i}=0$ for $i \in \{1, 2, \dots, n\}$,  $\hhat_{i,n}=\hhat_{n,i}=0,$ for $i \in \{1, 2, \dots, n-1\}$ and  $\abar=\zero$, we obtain
\begin{equation*}
    (e^i)^\top \hhat e^j=0.
\end{equation*}
Thus $\hhat=\zero_{n \times n}$. Therefore, the only solution to \eqref{eq:poisedqisys}  is the trivial solution.\smallskip\\

\noindent{\bf Case II: $\T=U_0.$}  The proof for this case is analogous to that of Case I. 
\end{proof}

It follows from the previous proposition that when a minimal poised set for the GSH is used,  the GSH  is equal to the simplex Hessian as described in \cite[Section 9.5]{conn2009introduction} (see also \cite[Section 3]{custodio2007}). Therefore, the results obtained  in \cite{custodio2007,conn2009introduction}  related to the simplex Hessian applies to the GSH. In this case, the GSH (and the simplex Hessian) is equal to the Hessian of the quadratic interpolation model passing through the sample points $\m(x^0;S,U_\ell).$ Therefore, the results developed in \cite{conn2009introduction,conn2008bgeometry,custodio2007} related to the simplex Hessian and the Hessian of a quadratic interpolation model are valid for the GSH.  One of the advantage of computing the GSH compared to the computation of the Hessian of the quadratic interpolation model has been discussed in Remark  \ref{rem:remark}. Another advantage of the GSH compared to the simplex Hessian described in \cite{conn2009introduction,custodio2007} is that it provides a simple explicit formula  that is well-defined as long as the matrices of directions employed are nonempty. Hence, the formula for the GSH  provides an approximation of the full Hessian, or a partial Hessian, for all possible cases defined in Definition \ref{def:cases}. 

Next, we provide an example that serves to show that a set of $(n+1)(n+2)/2$ distinct points in $\R^n$ that is poised for quadratic interpolation is not necessarily a minimal poised set for the GSH.

\begin{ex}\label{ex:nonminGSH}
Let $x^0=\begin{bmatrix} 0&0 \end{bmatrix}^\top$ be the point of interest. Consider $\mathcal{X}=$ $\{x^0=\zero,$ $ e^1,$ $ e^2,$ $ -e^1,$ $ -e^2,$ $ -e^1-e^2\}.$ Then $\mathcal{X}$ is poised for quadratic interpolation, but cannot be expressed as a minimal poised set for the GSH at $x^0$.
\end{ex}

\begin{proof}Using a similar approach as in  Proposition \ref{prop:minimalsetispoisedforqi}, one can verify that $\X$ is poised for quadratic interpolation. Now we show that $\X$ is not a minimal poised set for the GSH at $x^0$ using brute force. We need to build $S=\{ s^1, s^2\}$ such that the matrix corresponding to $S$ is full rank and $x^0 \oplus S \subseteq \X$. Hence, the possible choices for $S$ are 
\begin{align*}
    S &\in \left \{ \{e^1,e^2\},  \{e^1,-e^2\}, \{e^1,-e^1-e^2\}, \{-e^1,e^2 \}, \right . \notag\\ &\quad\quad\left . \{-e^1-e^2, e^2 \}, \{-e^1,-e^2 \}, \{ -e^1,-e^1-e^2 \},  \{-e^1-e^2, -e^2\}  \right \}.
\end{align*}
{\bf Case I: $S=\{e^1,e^2\}.$} In this case, we need to build $\overline T=\{t^1,t^2\}$ such that the matrix corresponding to $\overline T$ is full rank and 
    $$\{t^1, e^1+t^1,  e^2+t^1, t^2, e^1+t^2,  e^2+t^2\} = \X.$$
We see that the only possible choice of $t^1$ such that $\{t^1, e^1+t^1, e^2+t^1\} \subseteq \X $ is $t^1=-e^1-e^2$ (note $t^1 \neq \zero$ as we require full rank). However, the only possible choice of $t^2$ such that $\{t^2, e^1+t^2, e^2+t^2\} \subseteq \X $ is $t^2=-e^1-e^2$. As full rank implies $t^1$ cannot equal $t^2$, we see $S=\{e^1,e^2\}$ cannot provide the desired properties.
\smallskip\\
\noindent{\bf Cases II through VIII:} The other options for $S$ can be eliminated analogously.
\\
Therefore, $\X$ cannot be expressed as a minimal poised set for the GSH at $x^0$.
\end{proof}

Figure \ref{fig:notminset} shows all possible directions connecting two points in the set $\X$ from Example \ref{ex:nonminGSH}. If $\X$ were a minimal poised set for the GSH at $x^0=\zero$, then it would be possible to choose two directions (lines in blue) emerging from $x^0$ that connect to other points in $\Y$ and these same two directions would be emerging from two other points.
\begin{figure}[ht]
\caption{\centering A set that is poised for quadratic interpolation but not a minimal poised set for the GSH at $x^0$.}\label{fig:notminset}
\includegraphics[scale=0.5]{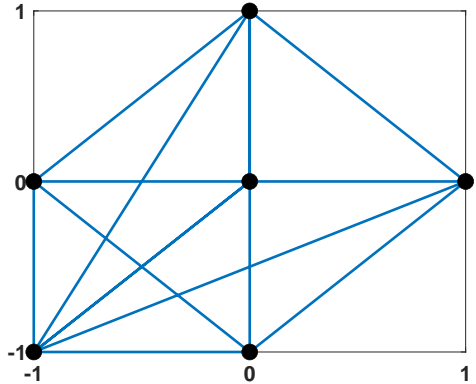}
\centering
\end{figure}

We conclude this section by providing formulae to obtain all the values of the coefficients involved in the quadratic interpolation function of $f$ over a minimal poised set for the GSH of the form $\m(x^0;S,U_\ell)$, which we denote by  $Q_f(x^0;S,U_\ell)(x).$

\begin{prop} Let $f:\dom f \subseteq \R^n \to \R$.  Let $\m(x^0;S,U_\ell) \subset \dom f$ be a minimal poised set for the GSH at $x^0$ as constructed in Proposition \ref{prop:reducingfe}. If $\ell \in\{1, 2, \dots, n\},$ then the Hessian matrix $\h$ of  the quadratic interpolation function $Q_f(x^0;S,U_\ell)(x)$    is given by $\h=S^{-\top}\hhat S^{-1},$ where the entries of the symmetric matrix $\hhat \in \R^{n \times n}$ are 
\begin{align*}
    \hhat_{i,\ell}&=-f(x^0+s^i-s^\ell)+f(x^0+s^i)+f(x^0-s^\ell)-f(x^0), \quad i \in \{1, 2, \dots, n\}\setminus\{\ell\},\\
    \hhat_{i,i}&=f(x^0+2s^i-s^\ell)-2f(x^0+s^i-s^\ell)+f(x^0-s^\ell), \quad i \in \{1, 2, \dots, n\}\setminus\{\ell\},\\
     \hhat_{\ell,\ell}&=f(x^0+s^\ell)+f(x^0-s^\ell)-2f(x^0),\\
    \hhat_{i,j}&=f(x^0+s^i+s^j-s^\ell)-f(x^0+s^i-s^\ell)-f(x^0+s^j-s^\ell)+f(x^0-s^\ell),
\end{align*}
for all $i,j \in \{1, 2, \dots, n\}\setminus\{\ell\}, i \neq j.$ If $\ell=0,$ then
\begin{align*}
    \hhat_{i,i}&=f(x^0+2s^i)-2f(x^0+s^i)+f(x^0), \quad i \in \{1, 2, \dots, n\},\\
    \hhat_{i,j}&=f(x^0+s^i+s^j)-f(x^0+s^i)-f(x^0+s^j)+f(x^0), \quad i,j \in \{1, 2, \dots, n\}, i \neq j.
\end{align*}
For all $\ell \in \{0, 1, \dots, n\},$ the vector $\alpha$ associated to $Q_f(x^0;S,U_\ell)(x)$ is given by  $\alpha=S^{-\top}\abar$, where  the entries of $\abar \in \R^n$ are 
\begin{align*}
    \abar_i&=f(x^0+s^i)-f(x^0)-\frac{1}{2}\hhat_{i,i}-(x^0)^\top\h s^i, \quad i \in \{1, 2,\dots, n\}.
\end{align*}
 The scalar $\alpha_0$ of $Q_f(x^0;S,U_\ell)(x)$ is 
\begin{align*}
    \alpha_0 &=f(x^0)-\alpha^\top x^0-\frac{1}{2}(x^0)^\top \h x^0.
\end{align*}
\end{prop}
\begin{proof}
The result is obtained by using Definition \ref{def:quadinterpolationfunc}. Let
\begin{align*}
    Q_f(x^0;S,U_\ell)(x)=\alpha_0+\alpha^\top x+\frac{1}{2}x^\top \h x,
\end{align*}
where $a_0 \in \R, \alpha \in \R^n$ and $\h=\h^\top \in \R^{n \times n}.$ Suppose $\ell \in \{1, 2, \dots, n\}.$
Evaluating $Q_f(x^0;S,U_\ell)(x)$ at $x^0$, we obtain
\begin{align}
    \alpha_0+\alpha^\top x^0+\frac{1}{2}(x^0)^\top\h x^0&=f(x^0). \label{eq:x0f}
\end{align}
Evaluating $Q_f(x^0;S,U_\ell)(x)$ at $x^0+s^i$ and using \eqref{eq:x0f}, we obtain
\begin{align}
    f(x^0)+\abar^\top e^i+(x^0)^\top \hba e^i+\frac{1}{2}\hhat_{i,i}&=f(x^0+s^i), \quad i \in \{1, \dots, n\}, \label{eq:x0psi}
\end{align}
where $\abar^\top=\alpha^\top S, \hba=\h S$ and $\hhat=S^\top \h S.$ Evaluating $Q_f(x^0;S,U_\ell)(x)$ at $x^0+s^i-s^\ell$ and using \eqref{eq:x0f} and \eqref{eq:x0psi}, we obtain
\begin{equation}
   \scalemath{1}{ f(x^0+s^i)-\abar^\top s^\ell-(x^0)^\top \hba e^\ell-\hhat_{i,\ell}+\frac{1}{2}\hhat_{\ell,\ell}=f(x^0+s^i-s^\ell),  i \in \{1, \dots, n\}\setminus \{\ell\}.}\label{eq:x0simsk}
\end{equation}
Evaluating $Q_f(x^0;S,U_\ell)(x)$ at $x^0-s^\ell$ and using \eqref{eq:x0f}, we find
\begin{align}
    -\abar^\top e^\ell-(x^0)^\top \hba e^\ell+ \frac{1}{2} \hhat_{\ell,\ell}&=f(x^0-s^\ell)-f(x^0). \label{eq:x0msk}
\end{align}
Substituting \eqref{eq:x0msk} into \eqref{eq:x0simsk}, we obtain
\begin{equation}
    \hhat_{i,\ell}=\hhat_{\ell,i}=-f(x^0+s^i-s^\ell)+f(x^0+s^i)+f(x^0-s^\ell)-f(x^0),  i \in \{1, \dots, n\}\setminus \{\ell\}. \label{eq:hhatik}
\end{equation}
Evaluating $Q_f(x^0;S,U_\ell)(x)$ at $x^0+2s^i-s^\ell$ and using \eqref{eq:x0f}, \eqref{eq:x0psi} and \eqref{eq:x0msk}, we find
\begin{align}
     \scalemath{0.9}{2f(x^0+s^i)-2f(x^0)+f(x^0-s^\ell)+\hhat_{i,i}-2\hhat_{i,\ell}=f(x^0+2s^i-s^\ell) \label{eq:x0p2simsk}, i \in \{1, \dots, n\}\setminus\{\ell\}.}
\end{align}
Substituting \eqref{eq:hhatik} in \eqref{eq:x0p2simsk}, we find
\begin{align*}
    \hhat_{i,i}&=f(x^0+2s^i-s^\ell)+f(x^0-s^\ell)-2f(x^0+s^i-s^\ell), \quad i \in \{1, 2, \dots, n\}\setminus\{\ell\}.
\end{align*}
Evaluating $Q_f(x^0;S,U_\ell)(x)$ at $x^0+s^i+s^j-s^\ell$ and using \eqref{eq:x0f}, \eqref{eq:x0psi} and \eqref{eq:x0msk}, we find
\begin{align*}
    \hhat_{i,j}=\hhat_{j,i}&=f(x^0+s^i+s^j-s^\ell)-f(x^0+s^i-s^\ell)-f(x^0+s^j-s^\ell)+f(x^0-s^\ell), 
\end{align*}
for $i,j \in \{1, 2, \dots, n\}\setminus\{\ell\}, i \neq j.$ Rearranging \eqref{eq:x0msk}, we get
\begin{align}
    \frac{1}{2} \hhat_{\ell,\ell}-f(x^0-s^\ell)+f(x^0)&=\abar^\top e^\ell+(x^0)^\top\hba e^\ell \label{eq:rearrangex0msk}.
\end{align}
Substituting \eqref{eq:rearrangex0msk} into \eqref{eq:x0psi} for $i=\ell$, we obtain
\begin{align*}
    \hhat_{\ell,\ell}&=f(x^0+s^\ell)+f(x^0-s^\ell)-2f(x^0).
\end{align*}
The entries of the vector $\abar$ are found by isolating $\abar^\top e^i$ in \eqref{eq:x0psi}. We obtain
\begin{align*}
    \abar_i&=f(x^0+s^i)-f(x^0)-\frac{1}{2} \hhat_{i,i}-(x^0)^\top\h s^i, \quad i \in \{1, 2, \dots, n\} 
\end{align*}
where $\alpha_i=S^{-\top}\abar_i.$ Lastly, the scalar $\alpha_0$ is obtained from \eqref{eq:x0f}. We find
\begin{align*}
    \alpha_0&=f(x^0)-\alpha^\top x^0-\frac{1}{2}(x^0)^\top \h x^0.
\end{align*}
If $\ell=0$, a similar process can be applied to obtain $\h, \alpha$ and $\alpha_0.$ \qedhere 
\end{proof}

Using $\T=U_\ell$, it is worth emphasizing that $Q_f(x^0;S,\T)(x)$ can be obtained for free in terms of function evaluations whenever  $\nabla^2_s f(x^0;S,\T)$  has already been computed. Indeed, all the coefficients of $Q_f(x^0;S,\T)$ are computed using the same function evaluations used for finding $\nabla_s^2 f(x^0;S,\T).$ Since $\mathcal{M}(x^0;S,\T)$ is poised for quadratic interpolation, there exists one and only one Hessian matrix $\h$. It follows that the Hessian $\h$ of the quadratic interpolation function $Q_f(x^0;S,\T)(x)$ must be equal to  $\nabla_s^2 f(x^0;S,\T).$ 

\section{Conclusion}\label{sec:conc}

We have presented the generalized simplex Hessian (GSH), a compact and simple formula for approximating the Hessian of a function $f$.  The GSH is well-defined as long as the matrices $S$ and $T_i$ are nonempty. In particular, the matrices $S$ and $T_i$ do not need to be full row rank nor square to compute the GSH.  We have also presented a centered version of the GSH called the generalized centered simplex Hessian (GCSH) and shown its relation to the GSH (Section \ref{sec:relation}).

In Theorems \ref{thm:mainprop} and \ref{thm:ebcentered2}, we developed error bounds for the GSH and GCSH.  Under some assumptions , the GSH is an order-1 accurate approximation of the partial Hessian $\Projuivi \nabla^2 f(x^0)$, and the GCSH is an order-2 accurate approximation of the partial Hessian $\Projuivi \nabla^2 f(x^0)$.  Indeed,  when all matrices of directions in the set $\Ti$ are equal, then we obtain error bounds that covers all possible  16 cases defined in Definition \ref{def:cases}.  When all matrices $T_i$ are not equal, then 12 out of 16 cases are covered.  Future research directions include an investigation of the cases not covered in Theorems \ref{thm:mainprop} and \ref{thm:ebcentered2}: the four cases $S$-nondetermined \textbackslash overdetermined and $\Ti$-nondetermined\textbackslash underdetermined.  

We note that when a matrix is nondetermined or overdetermined, it is possible to remove columns (or rows) of the matrix to make it full rank.  This may lead to error bounds for the nondetermined cases. 

As the results in this paper are reported in terms of order-$N$ accuracy, it is possible to apply the theory of DFO calculus developed in \cite{chen2022error} directly.  In \cite{chen2022error}, approximation based product, quotient, and chain rules are developed.  For each rule, it is shown that when an objective function can be decomposed into a product, quotient, or composition of simpler functions, then it is possible to construct approximate Hessians of the simpler functions and then use those to construct approximation Hessians of the objective function. Combining \cite{chen2022error} with the research herein has lead to a model-based trust region algorithm using the calculus-based approximation techniques \cite{Hare2023trdfa}.

In Section \ref{sec:qishc}, we showed that if a minimal poised set for the GSH is used, the number of  distinct function evaluations required to compute the GSH is $(n+1)(n+2)/2.$ We investigated the relationship between a minimal poised set for the GSH and poisedness for quadratic interpolation, proving that a minimal poised set for the GSH of the form $\m(x^0;S,U_\ell)$ is well-poised for quadratic interpolation, but that the converse does not necessarily hold. In this case, we know that $\sh f(x^0;S,\T)$ with $\T=U_\ell$ is equal to  the Hessian of the quadratic interpolation function $Q_f (x^0;S,U_\ell)$. We also developed explicit formulae for obtaining the parameters of the quadratic interpolation function of $f$ over a minimal poised set for the GSH of the form $\m(x^0;S,U_\ell)$. 

Future research should explore if the sets $U_0, U_1, \dots, U_n$ as defined in Proposition \ref{prop:reducingfe} are the only possible choices for $\T \in \R^{n \times n}$ such that $\s(x^0;S,\T),$ where $S \in \R^{n \times n}$ is full rank, is a minimal poised set for GSH at $x^0$.  It can be proven that it is indeed the case by using brute force in $\R$ and $\R^2,$ but it is still unclear how to generalized this claim in an arbitrary dimension $n$.

While it is immediately clear that a minimal poised set for GSH can be used to construct a `small' set for the GCSH (using Proposition \ref{prop:relationGCSHGSH}), it is not clear if such a set would be minimal.  Future research will investigate the properties of  a minimal poised set for the GCSH and the matrices of directions that can be used to minimize the number of sample points.

 It is known that the diagonal entries of the Hessian may be obtained for free in terms of function evaluations when a centered simplex gradient is computed. Hence, a future research direction could be  to verify if the main  diagonal entries of $\nabla^3 f(x^0)$ may be obtained for free  when a GCSH is computed.

%An obvious direction for the advancement of this research is the application of the Hessian approximation techniques within model-based DFO algorithms.  It will be important to develop techniques to reuse  previously evaluated sample points when implementing the Hessian approximation techniques introduced in this paper. Indeed, at least $(n+1)(n+2)/2$ function evaluations are required to compute an accurate generalized simplex Hessian and this might be too many function evaluations when $f$ is an expensive function. On the other hand, the matrix structure used in defining the GSH will make this Hessian approximation easy to implement within a trust-region derivative-free algorithm. Since the error bounds include the {\em over-determined} case (i.e., the case where more than $(n+1)(n+2)/2$ sample points are used), the GSH will also allow for some flexibility in reusing previously evaluated points without the need to clean the sample set.

On a final note, MATLAB implementations of the GSH and the GCSH are available upon request. 

\section*{Acknowledgements}
The authors express their deep gratitude to the anonymous referees for the comprehensive reviews that improved the quality of this work.

Hare's research is partially funded by the Natural Sciences and Engineering Research Council (NSERC) of Canada, Discover Grant \#2018-03865.  Jarry-Bolduc's research is  partially funded by the Natural Sciences and Engineering Research Council (NSERC) of Canada, Discover Grant \#2018-03865. Jarry-Bolduc would like to acknowledge UBC for the funding received through the University Graduate Fellowship award. 

\bibliographystyle{plain}
\bibliography{Bibliography}
\end{document}